\documentclass[11pt,a4paper,reqno]{amsart}
\usepackage[english]{babel}
\usepackage[T1]{fontenc}
\usepackage{palatino}
\usepackage{amsmath}
\usepackage{amssymb}
\usepackage{amsthm}
\usepackage{amsfonts}
\usepackage{graphicx}
\usepackage{color}
\usepackage{mathtools}

\usepackage[colorlinks = true, citecolor = black]{hyperref}
\title{Isometric Embeddings into Heisenberg Groups}
\author[Z.\ M.\ Balogh]{Zolt\'{a}n M.\ Balogh}
\address{Department of Mathematics and Statistics\\ University of Bern \\ Sidlerstrasse 5, 3012 Bern, Switzerland}
\email{zoltan.balogh@math.unibe.ch}
\author[K. F\"assler]{Katrin F\"assler}
\address{Department of Mathematics\\ University of Fribourg \\ Chemin du Mus\'{e}e 23,
CH-1700 Fribourg, Switzerland}
\email{katrin.faessler@unifr.ch}
\author[H.\ Sobrino]{Hernando Sobrino}
\address{Department of Mathematics and Statistics\\ University of Bern \\ Sidlerstrasse 5, 3012 Bern, Switzerland}
\email{hernando.sobrino@math.unibe.ch}
\date{\today}
\thanks{Z.B.\ was supported by the Swiss National Science Foundation through the project 165507 `Geometric Analysis of Sub-Riemannian Spaces'. K.F.\ was supported by the Swiss National Science Foundation through the grant 161299 `Intrinsic rectifiability and mapping theory on the Heisenberg group'.}

\newcommand{\E}{\mathbb{E}}

\def\Barint_#1{\mathchoice
          {\mathop{\vrule width 6pt height 3 pt depth -2.5pt
                  \kern -8pt \intop}\nolimits_{#1}}%
          {\mathop{\vrule width 5pt height 3 pt depth -2.6pt
                  \kern -6pt \intop}\nolimits_{#1}}%
          {\mathop{\vrule width 5pt height 3 pt depth -2.6pt
                  \kern -6pt \intop}\nolimits_{#1}}%
          {\mathop{\vrule width 5pt height 3 pt depth -2.6pt
                  \kern -6pt \intop}\nolimits_{#1}}}

\numberwithin{equation}{section}

\theoremstyle{plain}
\newtheorem{thm}[equation]{Theorem}

\newtheorem{lemma}[equation]{Lemma}

\newtheorem{ex}[equation]{Example}
\newtheorem{cor}[equation]{Corollary}
\newtheorem{proposition}[equation]{Proposition}

\theoremstyle{definition}

\newtheorem{definition}[equation]{Definition}

\theoremstyle{remark}
\newtheorem{remark}[equation]{Remark}

\begin{document}

\begin{abstract}
We study isometric embeddings of a Euclidean space or a Heisenberg group into a higher dimensional Heisenberg group, where both the source and target space are equipped with an arbitrary left-invariant homogeneous distance that is not necessarily sub-Riemannian. We show that if all infinite geodesics in the target are straight lines, then such an embedding must be a homogeneous homomorphism. We discuss a necessary and certain sufficient conditions for the target space to have this `geodesic linearity property', and we provide various examples.
\end{abstract}

\maketitle

\section{Introduction}

Isometries play a crucial role in metric geometry. It is a challenging task to decide whether two metric spaces (or subsets thereof) are isometric to each other. The task becomes more manageable if one can exploit additional structure on the space  to deduce a priori information on the form of isometries. For instance, according to a classical result by S.\ Mazur and S.\ Ulam, every isometry between normed vector spaces over $\mathbb{R}$ is affine, and this rigidity can be used to study which $\ell^n_p$ spaces are isometric.

In this paper we consider another class of metric spaces: Heisenberg groups $\mathbb{H}^n$ endowed with a homogeneous distance. By such a distance, we mean a left-invariant metric induced by a gauge function which is homogeneous with respect to
a one-parameter family of `Heisenberg dilations' adapted to the stratification of the underlying Lie algebra. An example is the Heisenberg group with its standard sub-Riemannian distance. It would go beyond the scope of this introduction to list the many motives for studying this particular space, but the interested reader can find more information for instance in the monograph \cite{MR2312336}. Surjective isometries between Heisenberg groups and more general sub-Riemannian manifolds have received considerable attention in recent years \cite{MR1078163,MR2028665,AM,MR2738530,CLD}.
In this paper we consider distances that are not necessarily sub-Riemannian and isometric embeddings that are not necessarily surjective. We prove the following result:
\begin{thm}\label{t;main1}
Let $\mathbb{H}^m$ and $\mathbb{H}^n$, $n\geq m$, be endowed with left-invariant homogeneous distances $d$ and $d'$, respectively. If every infinite geodesic in $(\mathbb{H}^n,d')$ is a line, then every isometric embedding $f:(\mathbb{H}^m,d) \to (\mathbb{H}^n,d')$ is the composition of a left translation and a homogeneous homomorphism.
\end{thm}
For the precise -- slightly more general -- statement, see Theorem \ref{t:main_Eucl_Heis}. We emphasize that this result is about isometric embeddings, which to the best of our knowledge have not been studied in this setting before. For surjective isometries, the conclusion of Theorem \label{t;main1} is known even without the assumption on the geodesics in the target space. Heisenberg groups with homogeneous distances are examples of metric Lie groups in the sense of V.\ Kivioja and E.\ Le Donne, and in \cite{KLD}, the authors proved a group version of the Mazur-Ulam theorem, stating that every isometry between nilpotent connected metric Lie groups is `affine', that is, the composition of a left translation and a group isomorphism.  An analogous result is known for isometries between open sets in nilpotent metric Lie groups with additional structure: Le Donne and A.\ Ottazzi showed in \cite{LDO} that every isometry between open subsets of sub-Riemannian, or more generally sub-Finsler, Carnot groups is the restriction of an affine map. It is the purpose of the present paper to establish similar conclusions for isometric embeddings into Heisenberg groups. Examples in Section \ref{s:NonLinearEmbedd} show that such embeddings need not be affine in general. This is only natural -- even in the setting of normed vector spaces a non-surjective isometry is not necessarily affine, however it is, if the norm in the target space is strictly convex.
In normed spaces, the linearity of geodesics is one of many equivalent ways to characterize strict convexity, and the corresponding property turns out to be sufficient for Heisenberg groups as well.

The proof of the main result proceeds by showing that an isometric embedding must map foliations given by certain vector fields in the source to analogous foliations in the target. Ideas in this spirit have been used to study global isometries before, for instance in \cite{MR1078163}.  Our proof of Theorem \ref{t;main1} is self-contained and elementary. Unlike proofs in  \cite{KLD} and
\cite{LDO}, it does not proceed via first establishing smoothness of isometries. Moreover, it applies in particular also to the situation where the homogeneous distance in the source space is not a length distance.

%

We introduce properties of homogeneous left-invariant metrics which imply that the associated geodesics are lines. This provides methods to establish that the assumptions of  Theorem \ref{t;main1} are satisfied, but we also hope it to be of independent interest. As an illustration, we consider a one-parameter family of norms $N_{p,a}$ on $\mathbb{H}^n$ which are related to the $\ell_p$-norms in Euclidean space and we prove that $N_{p,a}$ has the geodesic linearity property iff $p\in (1,+\infty)$, see Section \ref{s:exNorms}. The distance associated to $N_{p,a}$ for $p=\infty$ admits infinite geodesics which are not lines, and in fact there exist in this setting isometric embeddings which are non-linear. A specific example of such a non-linear embedding is the mapping
\begin{displaymath}
f:(\mathbb{H}^1,d_{N_{a,\infty}})\to (\mathbb{H}^2,d_{N_{a,\infty}}),\quad (x,y,t)\mapsto (x,\sin(x),y,0,t),
\end{displaymath}
see Proposition \ref{p:nonlinearEmbed}.

\noindent \textbf{Structure.} Section \ref{s:prelim} contains preliminaries.  Section \ref{s:sc} is devoted to notions of strict convexity; after reviewing the definition in normed spaces, we introduce various notions of strict convexity for left-invariant homogeneous distances on Heisenberg groups. As a first result, we find sufficient conditions and one necessary condition to ensure that a homogeneous distance on $\mathbb{H}^n$ has the `geodesic linearity property', that is, all infinite geodesics are lines. In Section \ref{s:main}, we prove Theorem \ref{t;main1}. We continue with examples of metrics and isometric embeddings in Section \ref{ex:Examples}. The paper is concluded with final comments in Section \ref{s:final}.

\medskip

\noindent \textbf{Acknowledgements.} We thank Enrico Le Donne for helpful discussions on the subject of this paper, in particular for giving us the initial impetus to look at the sub-Finsler distances induced by homogeneous norms. We also thank Rita Pini and Andrea Calogero for motivating conversations.

\section{Preliminaries}\label{s:prelim}

First we discuss in Section \ref{ss:TheHeisenberggroup} the Heisenberg group and the homogeneous distances which we will consider thereon.
In Sections \ref{s:ProjectedNorms} and  \ref{s:LengthOfCurves} we collect facts about homogeneous norms, sub-Finsler distances, and lengths of curves. This goes back to the work of Kor\'{a}nyi for the sub-Rie{\-}mannian distance. The considered properties are folklore knowledge even for more general homogeneous norms, and discussed in various places in the literature, for instance in \cite{zbMATH06363213} or \cite{LDNG}, and we do not claim novelty here. However, since we could not always find references which stated the results in the desired generality and since we sometimes follow a different approach, we decided to include the relevant results and proofs.  This also serves the purpose of introducing the concepts used later in Section \ref{s:StrictlyConvexGroups}, where we propose new definitions of strict convexity in Heisenberg groups. Readers familiar with the present material may wish to go directly to Section \ref{s:sc}.

\subsection{The Heisenberg group}\label{ss:TheHeisenberggroup}
The \emph{$n$-th Heisenberg group}  $\mathbb{H}^{n}$ is the set $\mathbb{R}^{2n}\times\mathbb{R}$ equipped with the
multiplication
\begin{displaymath}
(z,t)\ast(z',t'):=(z+z',t+t'+2\langle z,J_n z'\rangle), \text{ where } J_n = \begin{pmatrix}0&-E_n\\E_n&0\end{pmatrix}\in\mathbb{R}^{2n\times2n},
\end{displaymath}
and $E_n$ denotes the $(n\times n)$ unit matrix. Sometimes it is convenient to write in coordinates
\begin{displaymath}
z=(x_1,\ldots,x_n,y_1,\ldots,y_n).
\end{displaymath}
It can be easily verified
that $(\mathbb{H}^{n},\ast)$ satisfies all properties of a group
with neutral element $e:=(0,0)$ and inverse $(z,t)^{-1}:=(-z,-t)$.  Denoting the nonlinear term $\langle z,J_n z'\rangle$ by $\omega_{n} (z,z')$, we remark that this expression defines a skew-symmetric bilinear form on $\mathbb{R}^{2n}$, and that two elements $(z,t)$ and $(z',t')$ in $\mathbb{H}^{n}$ commute if and only if the term $\omega_n(z,z')$ is zero. Since this does not hold for all elements in $\mathbb{H}^{n}$ (for example $\omega_{n}(e_1,e_{n+1})=-1$ for the first and $(n+1)$-th standard unit vector in $\mathbb{R}^{2n}$), it turns out that the Heisenberg group is  non-abelian.

We can also identify the Heiseberg group $\mathbb{H}^{n}$ with $\mathbb{C}^{n}\times\mathbb{R}$, associating the element $z=(x_1,\ldots,x_n,y_1,\ldots,y_n)\in\mathbb{R}^{2n}$ with $\hat{z}(z):=(x_1+\mathrm{i}y_1,\ldots,x_n+\mathrm{i}y_n)\in\mathbb{C}^{n}$. Using this notation, the expression $\omega_n(z,z')$ takes the form $\mathrm{Im}(\langle \hat{z}(z),\hat{z}(z')\rangle)$, where $\langle\cdot,\cdot \rangle$ denotes the standard inner product on $\mathbb{C}^{n}$.
\begin{definition}
Let $\lambda>0$. The map $\delta_{\lambda}:\mathbb{H}^{n}\to\mathbb{H}^{n},\ (z,t)\mapsto(\lambda z,\lambda^{2}t)$ is called \emph{$\lambda$-dilation}.
\end{definition}
It can be easily verified that any $\lambda$-dilation defines a group isomorphism with inverse $\delta_{\lambda^{-1}}$. It plays an analogous role as the usual scalar multiplication in $\mathbb{R}^n$. To unify the notation in Euclidean spaces and Heisenberg groups, we will sometimes write $\delta_{\lambda}(x):= \lambda x$ for such scalar multiplication $\delta_{\lambda}: \mathbb{R}^n \to \mathbb{R}^n$.

\begin{definition}
Consider groups $\mathbb{G}_1, \mathbb{G}_2 \in \{(\mathbb{R}^n,+),(\mathbb{H}^n,\ast):\, n\in \mathbb{N}\}$ and associated one-parameter families of dilations $\delta^1_{\lambda}$, $\delta^2_{\lambda}$. A \emph{homogeneous homomorphism} $A: \mathbb{G}_1 \to \mathbb{G}_2$ is a group homomorphism that commutes with dilations, that is
\begin{align}
A(\delta^1_{\lambda}(p))=\delta^2_{\lambda}(A(p)),\ \forall \lambda >0, \forall p\in{G}_1.\nonumber
\end{align}
\end{definition}


\begin{lemma}\label{l:homohHomo}
A map $A:(\mathbb{R}^{m},+)\to(\mathbb{H}^{n},\ast)$ is a homogeneous homomorphism if
and only if there exists a matrix $T\in\mathbb{R}^{2n\times m}$ with $T^{t}J_nT=0$, such that
$A(z)=(Tz,0)$, for all $z\in\mathbb{R}^{m}$. A map $A:(\mathbb{H}^{m},\ast)\to(\mathbb{H}^{n},\ast)$ is a homogeneous homomorphism if
and only if there exist $a\in\mathbb{R}$ and a matrix $T\in\mathbb{R}^{2n\times2m}$ with $aJ_m=T^{t}J_nT$, such that
$A(z,t)=(Tz,at)$, for all $(z,t)\in\mathbb{H}^{m}$.
\end{lemma}

\begin{proof}
The automorphisms of $\mathbb{H}^n$ (as a topological group) are well known, see for instance Theorem 1.22 in \cite{folland2016harmonic} for their classification. An analogous argument yields the expression for homogeneous homomorphisms $A:\mathbb{R}^m \to \mathbb{H}^n$ or $A:\mathbb{H}^m \to \mathbb{H}^n$ for $m<n$.
\end{proof}

In this note we are particularly interested in homogeneous homomorphisms which are injective. According to the formula in Lemma \ref{l:homohHomo} a necessary condition for a homogeneous homomorphism $A:\mathbb{H}^{m}\to\mathbb{H}^{n}$ to be injective is that $a \neq 0$. Since $J_m$ is injective, it then follows further that necessarily $m\leq n$ and $\mathrm{rank}(T)= 2m$. Taking the determinant on both sides of the identity $aJ_m=T^{t}J_nT$ determines the constant $a$ in terms of $T$. This yields the following characterization:

\begin{lemma}
A map $A:\mathbb{H}^{m}\to\mathbb{H}^{n}$ is an injective  homogeneous homomorphism if and only if $m\leq n$, $A(z,t)=(Tz,at)$ for all $(z,t) \in \mathbb{H}^m$ with
\begin{displaymath}
T = \left\{\begin{array}{ll}\sqrt{a} B,&\text{if }a>0,\\ \sqrt{-a} B \tau_m,&\text{if }a<0,\end{array} \right.
\end{displaymath}
where $B$ is a symplectic matrix in the sense that $B^t J_n B = J_m$, $\tau_m = \begin{pmatrix} 0 & E_m\\ E_m &0 \end{pmatrix}$, and
\begin{displaymath}
a= \sqrt[2m]{ \det (T^t J_n T)}>0\quad\text{or}\quad a= - \sqrt[2m]{ \det (T^t J_n T)}<0.
\end{displaymath}
\end{lemma}

\begin{definition}
Let $(G,\ast)$ be a group with neutral element $e$. We say that a  \emph{norm} on $G$ is a map $N:G\to\mathbb{R}_{\geq0}$ that satisfies
\begin{align}
&i)\ N(g)=0\Leftrightarrow g=e,\ \forall g\in G,\nonumber\\
&ii)\ N(g^{-1})=N(g),\ \forall g\in G,\nonumber\\
&iii)\ N(g*g')\leq N(g)+N(g'),\ \forall g,g'\in G.\nonumber
\end{align}
\end{definition}

\begin{definition}
Let $(G,\ *)$ be a group. A metric $d:G\times G\to\mathbb{R}_{\geq0}$
is called \emph{left-invariant}, if for every $g_{o}\in G$, the map $L_{g_{o}}:(G,d)\to(G,d),\ g\mapsto g_{o}\ast g$ is an isometry, that is, $d(g_{o}\ast g,g_{o}\ast g')=d(g,g')$, for all $g,g'\in G$.
\end{definition}
Every norm $N:G\to\mathbb{R}_{\geq0}$ induces
a left-invariant metric $d_{N}:G\times G\to\mathbb{R}_{\geq0}$,
and vice versa. More precisely, we can establish the following bijection
\begin{align}
\{N:G\to\mathbb{R}_{\geq0}:\;N\ \text{is a norm}\}&\to\{d:G\times G\to\mathbb{R}_{\geq0}:\;d\ \text{is a left-invariant metric}\}\nonumber\\
N&\mapsto d_{N}:G\times G\to\mathbb{R}_{\geq0},\ (g,g')\mapsto N(g^{-1}*g'),\nonumber
\end{align}
\begin{align}
\{d:G\times G\to\mathbb{R}_{\geq0}:\;d\ \text{is a left-invariant metric}\}&\to\{N:G\to\mathbb{R}_{\geq0}:\;N\ \text{is a norm}\}\nonumber\\
d&\mapsto N_{d}:G\to\mathbb{R}_{\geq0},\ g\mapsto d(g,e).\nonumber
\end{align}

\begin{definition}
A norm $N:\mathbb{H}^{n}\to\mathbb{R}_{\geq0}$ on the Heisenberg group is called \emph{homogeneous} if
\begin{align}
\ N(\delta_{\lambda}(p))=\lambda N(p),\ \text{for all }\lambda>0,\ \text{for all } p\in\mathbb{H}^{n}.\nonumber
\end{align}
\nonumber
\end{definition}
It is easy to see that a norm $N$ on $\mathbb{H}^{n}$ is homogeneous if and only its associated left-invariant metric is homogeneous in the sense that $d_N(\delta_{\lambda}(p),\delta_{\lambda}(q))=\lambda d_N (p,q)$. 
Every left-invariant distance on $\mathbb{H}^n$ induced by a homogeneous norm is a homogeneous distance in the sense of \cite[Definition 2.20]{LDR}.
From now on, we will use the expression "\emph{homogeneous distance on $\mathbb{H}^{n}$}" to talk about the left-invariant metric induced by a homogeneous norm. It follows from \cite[Proposition 2.26]{LDR} that the topology induced by any homogeneous distance on $\mathbb{H}^{n}$ coincides with the Euclidean topology on $\mathbb{R}^{2n+1}$, and from \cite[Corollary 2.28]{LDR} that any homogeneous norm is  continuous with respect to the Euclidean topologies of $\mathbb{R}^{2n+1}$ and $\mathbb{R}$.
In particular, we note that any two homogeneous distances on $\mathbb{H}^{n}$ induce the same topology.
In fact, once the homogeneous distances are known to be continuous with respect to the standard topology on $\mathbb{R}^{2n+1}$ one can show by a standard argument the even stronger fact that they are bi-Lipschitz equivalent. This is well known and can be found for instance in \cite[Lemma 1]{goodman1977}. On the other hand, the metric structure induced by a homogeneous norm $N$ on $\mathbb{H}^n$ is very different from $\mathbb{R}^{2n+1}$ endowed with the Euclidean distance $d_{eucl}$. The two distances $d_N$ and $d_{eucl}$ are not bi-Lipschitz equivalent for any choice of homogeneous norm $N$ on $\mathbb{H}^n$, however, one has that the identity map $(\mathbb{H}^n,d_N)\to (\mathbb{R}^{2n+1},d_{eucl})$ is locally Lipschitz.

\subsection{Projected norms}\label{s:ProjectedNorms}

Certain properties of a homogenous norm on $\mathbb{H}^n$ are encoded by its `projection' to $\mathbb{R}^{2n}\times \{0\}$. Our starting point is the following observation, which relates a homogeneous norm to a norm (in the classical sense of the word) in Euclidean space.

\begin{proposition}\label{p:ProjNormIsNorm}
For every homogeneous norm $N$ on $\mathbb{H}^n$, the function
\begin{displaymath}
\|\cdot\|: \mathbb{R}^{2n}\to [0,+\infty),\quad \|z\|:= N((z,0)).
\end{displaymath}
defines a norm on $\mathbb{R}^{2n}$.
\end{proposition}

\begin{proof}
Homogeneity and positive definiteness of $\|\cdot\|$ follow immediately from the corresponding properties of $N$. (Recall that the Heisenberg dilation acts like the usual scalar multiplication on points in $\mathbb{R}^{2n}\times \{0\} \subset  \mathbb{H}^n$.) The triangle inequality for $\|\cdot\|$ is based on the fact that $N((z,0))\leq N((z,t))$ for all $(z,t)\in \mathbb{H}^n$, which we record in Lemma \ref{l:normHorizVert}. Taking this for granted, we obtain
\begin{align*}
\|z+w\| &:= N((z+w,0)) \\&\leq N((z+w,2\omega_n(z,w))) = N((z,0)\ast (w,0))\leq N((z,0))+ N((w,0))\\& = \|z\| + \|w\|,
\end{align*}
for all $z,w \in \mathbb{R}^{2n}$, which concludes the proof.
\end{proof}

\begin{lemma}\label{l:normHorizVert}
If $N$ is a homogeneous norm on $\mathbb{H}^n$, then
\begin{displaymath}
N((z,0))\leq N((z,t)),\quad\text{for all }(z,t)\in\mathbb{H}^n.
\end{displaymath}
\end{lemma}

\begin{proof}
Consider an arbitrary point $(z,t)$ in $\mathbb{H}^n\setminus \{(0,0)\}$. We will show that
\begin{equation}\label{eq:estimateNorm_n}
N\left(\left(z,\frac{t}{2^n}\right)\right) \leq N((z,t)),\quad \text{for all }n\in\mathbb{N}.
\end{equation}
To see why this holds for $n=1$, we rely on the homogeneity and triangle inequality, which yield
\begin{align*}
2 N\left(\left(z,\tfrac{t}{2}\right)\right) \leq N\left(\left(2z,2t\right)\right)= N\left((z,t)\ast (z,t)\right)\leq 2 N(z,t).
\end{align*}
Dividing both sides of the inequality by $2$ yields \eqref{eq:estimateNorm_n} for $n=1$. The estimate \eqref{eq:estimateNorm_n} follows inductively. By continuity of $N$ it then follows that
\begin{displaymath}
N((z,0)) = \lim_{n\to \infty} N\left(\left(z,\frac{t}{2^n}\right)\right) \leq  N((z,t)),
\end{displaymath}
as desired.
\end{proof}

\subsection{Length of curves}\label{s:LengthOfCurves}

Different homogeneous norms on $\mathbb{H}^n$ can yield the same norm $\|\cdot\|$ on $\mathbb{R}^{2n}$, defined as in Proposition \ref{p:ProjNormIsNorm}. The probably best known examples for this phenomenon are the Kor\'{a}nyi norm (Example \ref{t:KoranyiNorm}) and the gauge function induced by the standard sub-Riemannian distance on the Heisenberg group. Even though different norms $N$ and $N'$ induce different distance functions $d_N$ and $d_{N'}$, rectifiable curves have the same length with respect to either metric provided that $N$ and $N'$ project to the same norm $\|\cdot\|$. In order to show this, let us recall that
the \emph{length} $L_d(\gamma)=L(\gamma)$ of a curve $\gamma: [a,b] \to (X,d)$ in a metric space is the supremum of $\sum_{i=1}^k d(\gamma(s_{i-1}),\gamma(s_i))$ over all partitions $a=s_0 \leq s_1 \leq \ldots s_{k}= b$. To explain why the length of curves in $(\mathbb{H}^n,d_N)$ is determined by $\|\cdot\|$, we first recall some theory from abstract metric spaces, following the presentation in \cite{MR1835418}.

\begin{definition}\label{d:speed}Let $(X,d)$ be a metric space and consider a curve $\gamma: I \to X$. The \emph{speed of $\gamma$ at $s$} is defined as
\begin{displaymath}
v_{\gamma}(s):= \lim_{\varepsilon \to 0} \frac{d(\gamma(s),\gamma(s+\varepsilon))}{|\varepsilon|},
\end{displaymath}
provided that this limit exists.
\end{definition}

For a proof of the subsequent result, see Proposition 1.16 in \cite{MR2454453} (or Theorem 2.7.6 in \cite{MR1835418} for the special case of Lipschitz curves). Recall that a curve $\gamma:[a,b]\to (X,d)$ is \emph{absolutely continuous} if for every $\varepsilon>0$ there exists $\delta>0$ such that for every finite collection $\{(a_i,b_i):\; 1\leq i\leq k\}$ of disjoint intervals $(a_i,b_i)\subset [a,b]$ with $\sum_{i=1}^k b_i - a_i <\delta$ one has $\sum_{i=1}^k d(\gamma(a_i),\gamma(b_i))<\varepsilon$.

\begin{thm}\label{t:LengthMetricSpace}
For every absolutely continuous curve $\gamma:[a,b] \to (X,d)$ in a metric space the speed $v_{\gamma}(s)$ exists for almost every $s \in [a,b]$, and the length of $\gamma$ is given by the Lebesgue integral of the speed, that is
\begin{displaymath}
L(\gamma)= \int_a^b v_{\gamma}(s)\;\mathrm{d}s.
\end{displaymath}
\end{thm}

From this general result one recovers the well-known formula for the length of curves in a normed space.

\begin{ex}\label{ex:LengthNormed}
Let $(X,d)$ be $(\mathbb{R}^k,\|\cdot\|)$ for some choice of norm $\|\cdot\|$. Every absolutely continuous curve $\gamma: [a,b] \to (\mathbb{R}^k,\|\cdot\|)$ is absolutely continuous with respect to the Euclidean distance on $\mathbb{R}^k$ and hence differentiable almost everywhere. If $s\in [a,b]$ is such a point where $\dot{\gamma}(s)$ exists, then
\begin{displaymath}
v_{\gamma}(s) = \lim_{\varepsilon \to 0} \left\|\frac{\gamma(s+\varepsilon)-\gamma(s)}{\varepsilon} \right\| = \|\dot{\gamma}(s)\|
\end{displaymath}
exists
by the homogeneity and continuity of the norm. Hence
\begin{equation}\label{eq:length_normed}
L(\gamma) = \int_a^b \|\dot{\gamma}(s)\| \;\mathrm{d}s.
\end{equation}
\end{ex}

As a second application of Theorem \ref{t:LengthMetricSpace}, we compute the length of curves in $\mathbb{H}^n$ equipped with a homogeneous distance. This result is folklore; we shall include a proof for convenience.

\begin{proposition}\label{p:length_homogeneous}
Assume that $N$ is a homogeneous norm on $\mathbb{H}^n$, and let $\gamma: [a,b] \to (\mathbb{H}^n,d_N)$ be a Lipschitz curve. We denote $\gamma=(\gamma_I,\gamma_{2n+1})$, so that $\gamma_{I}: [a,b] \to \mathbb{R}^{2n}$ is the projection of $\gamma$ to $\mathbb{R}^{2n}\times \{0\} \subset \mathbb{H}^n$.
 Then the length of $\gamma$ with respect to $d_N$ is given by
\begin{displaymath}
L(\gamma) = \int_a^b \|\dot{\gamma}_I(s)\|\;\mathrm{d}s,
\end{displaymath}
where $\|\cdot\|$ is the norm on $\mathbb{R}^{2n}$ induced by $N$ as in Proposition \ref{p:ProjNormIsNorm}.
\end{proposition}

The proof of this proposition is a rather immediate corollary of Theorem \ref{t:LengthMetricSpace} if one makes use of the theory of horizontal curves. A \emph{horizontal curve} in $\mathbb{H}^n$ is an absolutely continuous curve $\gamma:[a,b] \to \mathbb{R}^{2n+1}$ with the property that
\begin{displaymath}
\dot{\gamma}(s) \in H_{\gamma(s)}, \quad\text{for almost every }s\in [a,b],
\end{displaymath}
where for $p\in \mathbb{H}^n$, we set
\begin{displaymath}
H_p:= \mathrm{span}\left\{X_{1,p},\ldots,X_{n,p},Y_{1,p}\ldots,Y_{n,p} \right\}.
\end{displaymath}
Here $X_i$ and $Y_i$, $i=1,\ldots,n$, are the left-invariant vector fields (with respect to $\ast$) which at the origin agree with the standard basis vectors: $X_{i,0}=e_i$ and $Y_{i,0}= e_{n+i}$.
Denoting the $(2n+1)$ components of an absolutely continuous curve $\gamma:[a,b]\to \mathbb{H}^n$ by $\gamma_i$, $i=1,\ldots,2n+1$, it follows that $\gamma$ is horizontal if and only if
\begin{equation}\label{eq:horiz_comp}
\dot{\gamma}_{2n+1}(s) = 2 \sum_{i=1}^n \dot{\gamma}_i(s) \gamma_{n+i}(s)-\dot{\gamma}_{n+i}(s)\gamma_i(s),\quad\text{for almost every }s\in [a,b].
\end{equation}
It is well known that a horizontal curve $\gamma:[a,b]\to \mathbb{H}^n$ is rectifiable and admits a Lipschitz parametrization (see for instance \cite[Proposition 1.1]{MR3417082} for a proof and  note that this statement holds for any homogeneous norm on $\mathbb{H}^n$).
 In converse direction, every rectifiable curve admits a $1$-Lipschitz parametrization and this parametrization is horizontal, see \cite{MR979599}.

Curves in $\mathbb{H}^n$ which are Lipschitz with respect to a homogeneous distance can be differentiated almost everywhere not only in the usual, Euclidean, sense, but also in the sense of Pansu \cite{MR979599}, as a consequence of a far more general result concerning mappings between Carnot groups. If it exists, the Pansu differential of a curve $\gamma:[a,b] \to \mathbb{H}^n$ at a point $s\in [a,b]$ is a homogeneous homomorphism $D\gamma(s):\mathbb{R} \to \mathbb{H}^n$, given by
\begin{displaymath}
 D\gamma(s) r = \lim_{\varepsilon \to 0} \delta_{\frac{1}{\varepsilon}} \left(\gamma(s)^{-1}\ast \gamma(s+\varepsilon)\right) r.
\end{displaymath}
If $\gamma$ is at the same time differentiable at $s$ in the usual sense, then
\begin{equation}\label{eq:Pansu_diff}
D\gamma(s) r = \begin{pmatrix}\dot{\gamma}_1(s)\\\vdots\\\dot{\gamma}_{2n}(s)\\0\end{pmatrix}r.
\end{equation}

With this information at hand, we can proceed to the proof of Proposition \ref{p:length_homogeneous}.

\begin{proof}[Proof of Proposition \ref{p:length_homogeneous}] Since $\gamma$ is Lipschitz, it is a horizontal curve.
 Let $s \in [a,b]$ be a point in which $\gamma$ is differentiable in the usual sense and in the sense of Pansu (according to the discussion above, almost every point in $[a,b]$ is such a point).
From these assumptions, the homogeneity of the norm $N$, and the formula \eqref{eq:Pansu_diff} it follows that the speed of $\gamma$ exists at $s$ in the sense of Definition \label{d:speed} and is given by
\begin{align*}
v_{\gamma}(s)
:= N\left( \lim_{\varepsilon \to 0}\left(\delta_{\frac{1}{\varepsilon}}(\gamma(s)^{-1}\ast \gamma(s+\varepsilon))\right) \right)
 = N\left(\left(\dot{\gamma}_I(s),0\right)\right).
\end{align*}
Here, $\gamma_I:=(\gamma_1,\ldots,\gamma_{2n})$.
Inserting this expression into the formula for the length in Theorem \ref{t:LengthMetricSpace} completes the proof of the proposition.
\end{proof}

Proposition \ref{p:length_homogeneous} shows that in order to study the length of curves with respect to a homogeneous distance $d_N$, it suffices to consider the curves with respect to the sub-Finsler distance induced by the norm $\|\cdot\|$ which is associated to $N$ as in Proposition \ref{p:ProjNormIsNorm}. This has been observed in \cite{LDNG} for the first Heisenberg group (see the remark below Proposition 6.2 in \cite{LDNG}, where this is formulated in terms of the projection of the unit ball to the $(x,y)$-plane).

%
%
%

\begin{definition}
Given a norm $\|\cdot\|$ on $\mathbb{R}^{2n}$, the \emph{sub-Finsler distance} associated to $\|\cdot\|$ on $\mathbb{H}^n$ is the distance given by
\begin{displaymath}
d_{SF}(p,q):= \inf_{\gamma} \int_a^b \|\dot{\gamma}_I(s)\| \;\mathrm{d}s,
\end{displaymath}
where the infimum is taken over all horizontal curves $\gamma=(\gamma_{I},\gamma_{2n+1}): [a,b] \to \mathbb{H}^n$ with $\gamma(a)=p$ and $\gamma(b)=q$.
\end{definition}

Since $\|\cdot\|$, as a norm on $\mathbb{R}^{2n}$ is comparable to the Euclidean norm, it follows that $d_{SF}$ is comparable to the standard sub-Riemannian distance on $\mathbb{H}^n$, in particular, it is finite and positive. Clearly, $d_{SF}$ also satisfies the triangle inequality. Since left-translation is a bijection which sends horizontal curves to horizontal curves, preserving $\|\dot{\gamma}_I\|$, it follows further that $d_{SF}$ is left-invariant. Finally, it is homogeneous since $\|\cdot\|$ is homogeneous with respect to scalar multiplication, and Heisenberg dilations preserve horizontality of curves.

A particular role will be played in the following by geodesics with respect to $d_{SF}$. By a \emph{geodesic} $\gamma:I \to (X,d)$ in a metric space, we mean an isometric embedding of $I=[a,b]$ or $I=\mathbb{R}$ into $(X,d)$, that is,
\begin{displaymath}
d(\gamma(s),\gamma(s'))=|s-s'|,\quad\text{for all }s,s'\in I.
\end{displaymath}
If we have  $I=\mathbb{R}$ in the above definition, we say that $\gamma$ is an \emph{infinite geodesic}.
We stress that in Riemannian or sub-Riemannian geometry the word ``geodesic'' is also used with a different meaning, see for instance the discussion in \cite[Remark 1]{liu1995shortest}.

\begin{lemma}\label{l:proj_norm_agree}
Let $\|\cdot\|$ be a norm on $\mathbb{R}^{2n}$ and define $d_{SF}$ to be the associated sub-Finsler distance. Then
\begin{displaymath}
d_{SF}((z,0),(0,0))= \|z\|,\quad \text{for all }z\in\mathbb{R}^{2n}.
\end{displaymath}
\end{lemma}

\begin{proof}
Let $\gamma_I:[0,\|z\|]\to (\mathbb{R}^{2n},\|\cdot\|)$ be the geodesic which parametrizes the line segment that joins $0$ and $z$ in $\mathbb{R}^{2n}$, and note that $\gamma:=(\gamma_I,0):[0,\|z\|]\to \mathbb{H}^n$ is a Lipschitz continuous horizontal curve. Thus we find
\begin{displaymath}
d_{SF}((z,0),(0,0))  \leq \int_0^{\|z\|} \|\dot{\gamma}_I(s)\| \;\mathrm{d}s = L_{\|\cdot\|}(\gamma_I)=\|z\|.
\end{displaymath}
On the other hand, by definition of $d_{SF}$ and Example \ref{ex:LengthNormed}, we find that
\begin{displaymath}
d_{SF}((z,0),(0,0))  \geq \inf_{\sigma} \int_a^b \|\dot{\sigma}(s)\| \;\mathrm{d}s \geq \|z-0\|=\|z\|,
\end{displaymath}
where the infimum is taken over all absolutely continuous curves $\sigma:[a,b]\to \mathbb{R}^{2n}$ connecting $0$ and $z$.
\end{proof}

We wish to compare geodesics in $(\mathbb{H}^n,d_{SF})$ with geodesics for any homogeneous norm $N$ that induces $\|\cdot\|$. To do so, the subsequent characterization is useful.

\begin{lemma}\label{l:geodChar}
Let $(X,d)$ be a metric space. For a curve $\gamma:[a,b]\to X$ the following conditions are equivalent:
\begin{enumerate}
\item\label{i:geod1} $\gamma$ is a geodesic with respect to $d$, that is, $d(\gamma(s),\gamma(s'))=|s-s'|$ for all $s,s'\in [a,b]$,
\item\label{ii:geod2} $L(\gamma)=d(\gamma(a),\gamma(b))$ and $\gamma$ is parameterized by arc-length.
\end{enumerate}
\end{lemma}

This is well known; see for instance \cite[Remark 1.22]{MR1744486}.

%
%
%
%

\begin{proposition}\label{p:geodesics} Assume that $N$ is a homogeneous distance on $\mathbb{H}^n$ and let $\|z\|:= N((z,0))$. Denote by $d_{SF}$ the sub-Finsler distance associated to $\|\cdot\|$.
Let $I= [a,b]$ or $I=\mathbb{R}$. If $\gamma:I \to \mathbb{H}^n$ is a geodesic with respect to $d_N$, then it is also geodesic with respect to $d_{SF}$.
\end{proposition}

\begin{proof}
Let $\gamma:[s,s']\to (\mathbb{H}^n,d_N)$ be  geodesic. We claim that
\begin{equation}\label{eq:geodesic_claim}
d_{SF}(\gamma(s),\gamma(s'))= \int_s^{s'} \|\dot{\gamma}_I(\xi)\| \;\mathrm{d}\xi.
\end{equation}
If this is shown then it follows by Proposition \ref{p:length_homogeneous}, Lemma \ref{l:geodChar} and the geodesic assumption on $\gamma$ that
\begin{align*}
d_{SF}(\gamma(s),\gamma(s'))&= \int_s^{s'} \|\dot{\gamma}_I(\xi)\| \;\mathrm{d}\xi=L_{d_N}(\gamma|_{[s,s']})= d_N(\gamma(s),\gamma(s'))= |s-s'|.
\end{align*}
Since this holds for arbitrary $s<s'$ in $I$, it then follows that $\gamma$ is a geodesic with respect to $d_{SF}$. It remains to establish \eqref{eq:geodesic_claim}. Assume towards a contradiction that there exists a horizontal curve $\lambda:[t,t']\to \mathbb{H}^n$, connecting $\gamma(s)$ and $\gamma(s')$ such that
\begin{displaymath}
 \int_t^{t'} \|\dot{\lambda}_I(\xi)\| \;\mathrm{d}\xi <  \int_s^{s'} \|\dot{\gamma}_I(\xi)\| \;\mathrm{d}\xi.
\end{displaymath}
The curve  $\lambda$ is a priori only horizontal and thus absolutely continuous as a map to $\mathbb{R}^{2n+1}$,
but the horizontality ensures that it admits a Lipschitz reparametrization $\widetilde{\lambda}:[\widetilde{t},\widetilde{t}']\to (\mathbb{H}^n,d_N)$; see for instance \cite[Proposition 1.1]{MR3417082}. Hence
\begin{align*}
d_N(\gamma(s),\gamma(s'))\leq L_{d_N}(\lambda) = \int_{\widetilde{t}}^{\widetilde{t}'} \|\dot{\widetilde{\lambda}}_I(\xi)\|\;\mathrm{d}\xi= L_{d_{SF}}(\widetilde{\lambda})  \leq \int_t^{t'} \|\dot{\lambda}_I(\xi)\|\;\mathrm{d}\xi,
\end{align*}
where we have used in the last step that $\lambda$ is admissible in the definition of $d_{SF}$.
Hence we conclude
\begin{align*}
d_N(\gamma(s),\gamma(s'))&\leq \int_t^{t'} \|\dot{\lambda}_I(\xi)\|\;\mathrm{d}\xi <  \int_s^{s'} \|\dot{\gamma}_I(\xi)\| \;\mathrm{d}\xi= d_N(\gamma(s),\gamma(s')),
\end{align*}
which is a contradiction.
\end{proof}

\section{Notions of strict convexity}\label{s:sc}


We begin this section by reviewing the notion of strict convexity in normed vector spaces. Strictly convex norms can be characterized in many different ways, for instance through the shape of spheres or of geodesics in the space.  There exist natural counterparts of these properties for Heisenberg groups with a homogeneous left-invariant distance, which we introduce in Section \ref{s:StrictlyConvexGroups}. We show later in Section \ref{s:exNorms} that in this setting the properties cease to be all equivalent.

\subsection{Strictly convex norms on vector spaces}\label{s:StrictlyConvexNormed}

Strictly convex normed vector spaces play an important role as a class of spaces which are more flexible than inner product spaces, and still have better properties than arbitrary normed spaces. Various equivalent definitions of strict convexity for normed spaces are used concurrently in the literature. Proposition 7.2.1 in \cite{papadopoulos2005metric}, for instance, lists as many as nine different characterizations. We put our focus here on those three properties for which we will later formulate counterparts in the Heisenberg group.

\begin{proposition}\label{p:strict_convex_normed}
The following properties of a normed vector space $(V,\|\cdot\|)$ are equivalent:
\begin{enumerate}
\item\label{i:SC} \emph{strict convexity of the norm}:\\ if $v,w\in V\setminus \{0\}$ are such that $\|v+w\|=\|v\|+ \|w\|$, then $v=\lambda w$ for some $\lambda>0$,
\item\label{i:MP} \emph{midpoint property}:\\if $v,v_1,v_2\in V$ are such that $\|v_1-v\|=\|v_2-v\|=\frac{1}{2}\|v_1-v_2\|$, then $v=\frac{v_1+v_2}{2}$,
\item\label{i:GLP} \emph{geodesic linearity property}:\\ every infinite geodesic in $(V,\|\cdot\|)$ is a line in $V$.
\end{enumerate}
\end{proposition}

Equivalent characterizations of strict convexity as in Proposition \ref{p:strict_convex_normed} are well know; see for instance \cite[(i)]{vai:mazur-ulam} for the equivalence between \eqref{i:SC} and \eqref{i:MP}, and \cite{papadopoulos2005metric} for the fact that strict convexity is equivalent to $(V,\|\cdot\|)$ being \emph{uniquely geodesic}, which in turn implies \eqref{i:GLP}. For the convenience of the reader we include a proof for the fact that \eqref{i:GLP} implies \eqref{i:SC}.

\begin{proof}[Proof of  \eqref{i:GLP} $\Rightarrow$ \eqref{i:SC}] It turns out that the linearity of infinite geodesics is sufficient to establish strict convexity. To see this, consider arbitrary $v,w\in V\setminus \{0\}$ with the property that $\|v+w\|=\|v\|+ \|w\|$. Using these particular points, we construct an infinite geodesic, namely $\gamma:\mathbb{R}\to V$, defined by
\begin{displaymath}
\gamma(s):= \left\{\begin{array}{ll}\frac{v}{\|v\|}s,&s\in (-\infty,0]\\
\frac{w}{\|w\|}s,&s\in (0,+\infty).
\end{array} \right.
\end{displaymath}
It is clear that $\gamma$ restricted to $(-\infty,0]$ and $(0,+\infty)$ is geodesic. In order to verify that $\gamma$ is globally geodesic, we first show that if $\|v+w\|= \|v\|+\|w\|$, then there are in fact plenty of points with this property. Indeed, for arbitrary $a,b\in (0,1)$, we find
\begin{align*}
\|v\|+\|w\|&= \|v+w\|\\
&\leq \|a v + bw\| + (1-b)\|w\| + (1-a) \|v\|\\
&\leq a\|v\| + b\|w\| + (1-b)\|w\| + (1-a) \|v\|\\
&= \|v\|+ \|w\|.
\end{align*}
This shows that in every step of the above chain of estimates equality must be realized and thus
\begin{displaymath}
\|av + b w\| = a\|v\| + b \|w\|,\quad \text{for all }a,b\in [0,1].
\end{displaymath}
By scalar multiplication we deduce that the same identity holds for all $a,b\geq 0$.  This can be employed to prove that $\gamma$ is a global geodesic. To this end, it suffices to observe for $s\in (-\infty,0]$ and $s'\in (0,+\infty)$ that
\begin{align*}
\|\gamma(s')-\gamma(s)\| &= \left\|\frac{w}{\|w\|}s' + \frac{v}{\|v\|}(-s) \right\|= s'-s
= |s'-s|,
\end{align*}
Since $\gamma$ is therefore an infinite geodesic with $\gamma(0)=0$, it follows by the geodesic linearity property that $\gamma$ must be of the form $\gamma(s)= us$, $s\in \mathbb{R}$, for a vector $u\in V$ with $\|u\|=1$. We conclude that
 $v= (\|v\|/\|w\|)w$. The same argument applies to all such pairs of points $v$ and $w$, which shows that $(V,\|\cdot\|)$ is strictly convex. This concludes the proof of the proposition.
\end{proof}

Strictly convex norms have found many applications, some of which are listed for instance in
\cite{zbMATH03556688}. The most relevant result in the context of the present paper is the following.

\begin{thm}
Assume that $(V,\|\cdot\|_V)$ and $(W,\|\cdot\|_W)$ are two $\mathbb{R}$ vector spaces. If $\|\cdot\|_W$ is strictly convex, then every isometric embedding $f:(V,\|\cdot\|_V) \to (W,\|\cdot\|_W)$ is affine.
\end{thm}

As explained in \cite{vai:mazur-ulam}, this theorem follows from the midpoint property of strictly convex norms together with the fact that a continuous map $f:(V,\|\cdot\|_V) \to (W,\|\cdot\|_W)$ is affine if it preserves midpoints of line segments, see Lemma
\ref{l:linearity} below.

\subsection{Strictly convex norms on Heisenberg groups}\label{s:StrictlyConvexGroups}

\subsubsection{Notions of strict convexity}\label{ss:NotionsOfStrictConvexity}
We saw in Section \ref{s:StrictlyConvexNormed} that strict convexity in normed real vector spaces has different equivalent formulations. One of these formulations, defined as \emph{geodesic linearity property} in Proposition \ref{p:strict_convex_normed}, \eqref{i:GLP}, can be generalized: we say that any real linear space equipped with a metric has the \emph{geodesic linearity property}, if every infinite geodesic is a line. When the metric is induced by a norm, then this property is equivalent to strict convexity.


Our aim is to relate the geodesic linearity property for the Heisenberg group equipped with a homogeneous distance to \emph{strict convexity} and \emph{midpoint property}, defined in an intuitively analogous way.



When we speak about lines in the Heisenberg group $\mathbb{H}^n$, we mean lines in the underlying vector space $\mathbb{R}^{2n+1}$.
Straight lines $l(s)=p_0+sv,v\neq0,$ in $\mathbb{R}^{2n+1}$ can be also written as $l(s)=p_0\ast (sz,st)$ for an appropriate $(z,t)\in\mathbb{H}^{n}\backslash\{0\}$. We call $l$ \emph{horizontal} if $t=0$, and \emph{non-horizontal} if $t\neq0$. From now on, we use the word \emph{line} to talk about the curve, or its image.

\begin{proposition}\label{p:horiz_line}
 Let $l:\mathbb{R}\to\mathbb{H}^{n},\ s\to p_0 *(sz,st)$, be a straight line and $N$ a homogeneous norm on $\mathbb{H}^{n}$. If $l$ is horizontal, then it can be reparameterized to be an infinite geodesic on the metric space $(\mathbb{H}^{n},d_N)$. If $l$ is non-horizontal, the segment determined by any two different points of its image is not rectifiable.
  \end{proposition}

  Proposition \ref{p:horiz_line} is well known and follows from more general results about rectifiable curves (see for example \cite{MR979599} and the discussion around \eqref{eq:horiz_comp}).

 \begin{definition}
Let $N$ be an homogeneous norm on $\mathbb{H}^n$. We say that $N$ is \emph{horizontally strictly convex} if for all $p,p'\neq e$ it holds
\begin{align}
&N(p*p')=N(p)+N(p')\Rightarrow p,p' \text{ lie on a horizontal line through the origin, i.e.,}\nonumber\\
 &\exists z\in\mathbb{R}^{2n}\backslash\{0\},s,s'\in\mathbb{R},\ \text{such that}\ p=(sz,0)\ \text{and}\ p'=(s'z,0).\nonumber
\end{align}
\end{definition}
\begin{lemma}
The following two conditions for a homogeneous norm $N$ on $\mathbb{H}^{n}$ are equivalent:
\begin{align}
&(1)\ N\ \text{is horizontally strictly convex},\nonumber\\
&(2)\ \text{for all}\ p_1,p_2,p\in\mathbb{H}^{n}, p_1\neq p, p_2\neq p,\ \text{with}\ d_N (p_1,p_2)=d_N (p_1,p)+d_N (p,p_2),\nonumber\\
&\text{the points}\ p_1,p_2\ \text{belong to the horizontal line}\ l:=\{p*(sz,0):\;s\in\mathbb{R}\}\ \text{for some}\ z\in\mathbb{R}^{2n}\backslash\{0\}.\nonumber
\end{align}
\end{lemma}

\begin{proof}
First, we prove the implication $(1)\Rightarrow(2)$. For this, consider $p_1,p_2,p\in\mathbb{H}^{n}$, with $p_1$ and $p_2$ both distinct from $p$, satisfying
\begin{align}
 d_N (p_1,p_2)=d_N (p_1,p)+d_N (p,p_2).\nonumber
 \end{align}
 Defining $q:=(p^{-1}*p_2)^{-1}$ and $q':=p^{-1}*p_1$, it follows
 \begin{align*}
 N(q*q')=N((p_2)^{-1}*p_1)&=d_N (p_1,p_2)=d_N (p_1,p)+d_N (p,p_2)\\
 &=N(p^{-1}*p_1)+N(p^{-1}*p_2)=N(q')+N(q).
  \end{align*}
  Since $q\neq e$ and $q'\neq e$, the horizontally strictly convexity of $N$ implies that there exist $z\in\mathbb{R}^{2n}\backslash\{0\}$ and $s,s'\in\mathbb{R}$ such that $q=(sz,0)$ and $q'=(s'z,0)$. This implies $p_1=p*q'=p*(s'z,0)$ and $p_2=p*q^{-1}=p*(-sz,0)$, as desired.\newline\newline
  Now, for proving the implication $(2)\Rightarrow(1)$, consider $p,p'\in\mathbb{H}^{n}$, $p,p'\neq e$, that satisfy
  \begin{align}
  N(p*p')=N(p)+N(p').\nonumber
  \end{align}
  This implies
  \begin{align}
  d_N (p',p^{-1})=N(p*p')=N(p)+N(p')=d_N (e,p^{-1})+d_N (p',e).\nonumber
    \end{align}
  Since $p'\neq e$ and $p^{-1}\neq e$, by assumption there exist $z\in\mathbb{R}^{2n}\backslash\{0\}$ and $s,s'\in\mathbb{R}$ such that $p'=e(s'z,0)=(s'z,0)$ and $p=(p^{-1})^{-1}=(e(sz,0))^{-1}=(sz,0)^{-1}=(-sz,0)$, as desired.
  \end{proof}
 \begin{definition}[Midpoint property]
  Let $N$ be a homogeneous norm on $\mathbb{H}^n$. We say that $N$ has the midpoint property, if for all $p_1,p_2,q\in\mathbb{H}^n$ it holds
  \begin{align}
  d_N(p_1,p_2)=2d_N(p_1,q)=2d_N(p_2,q)\Rightarrow q=\frac{p_1+p_2}{2}.\nonumber
  \end{align}
  \end{definition}

  This notion is motivated by studying the behavior of homogeneous norms along horizontal lines.
  In the above definition, scalar multiplication and addition in the expression for $q$ are understood via the identification of $\mathbb{H}^n$ with $\mathbb{R}^{2n+1}$. The midpoint property is equivalent to the following metric condition:
  \begin{displaymath}
  d_N(p,p^{-1})=2 d_N(p,q)= 2 d_N(p^{-1},q)\quad\Rightarrow\quad q=e.
  \end{displaymath}

  \begin{proposition}
  Let $N$ be a horizontally strictly convex homogeneous norm on $\mathbb{H}^n$. Then $N$ has the midpoint property.
   \end{proposition}
     \begin{proof}
   Consider points $q,p_1,p_2\in\mathbb{H}^{n}$ with $d_N (p_1,p_2)=2d_N (p_1,q)=2d_N(p_2,q)$. Assuming without loss of generality that $q\neq p_i$ for $i=1,2$, we have
   \begin{align}
   d_N (p_1,p_2)\leq d_N(p_1,q)+d_N(q,p_2)=\tfrac{1}{2}d_N (p_1,p_2)+\tfrac{1}{2}d_N (p_1,p_2)=d_N (p_1,p_2).\nonumber\\
      \end{align}
      Since $N$ is horizontally strictly convex, there exist $z\in\mathbb{R}^{2n}\backslash\{0\}$ and $s_1,s_2\in\mathbb{R}$, such that $p_1=q*(s_1z,0)$ and $p_2=q*(s_2z,0)$. From this, we get
      \begin{align}
      &d_N (p_i,q)=N(q^{-1}*p_i)=N((s_i z,0))=|s_i|N(z,0),\nonumber\\
      &d_N(p_1,p_2)=d_N((s_1z,0),(s_2z,0))=N((-s_1z,0)*(s_2 z,0))=N((s_2-s_1)z,0)\nonumber\\
      &=|s_2-s_1|N(z,0).
      \end{align}
      Since $z\neq0$, $N(z,0)$ is a strictly positve number. From the last calculations, we obtain
      \begin{align}
      &d_N (p_1,p_2)=2d_N (q,p_1)=2d_N(q,p_2)\Leftrightarrow |s_2-s_1|N(z,0)=2|s_1|N(z,0)=2|s_2|N(z,0)\nonumber\\
      &\Leftrightarrow |s_2-s_1|=2|s_1|=2|s_2|.
            \end{align}
            The reader can convince her- or himself that the last equality implies $s_2=-s_1$.
                Finally, writing $q=(z_0,t_0)$ we get
                \begin{align}
                &\frac{p_1+p_2}{2}=\tfrac{1}{2}\left((z_0+s_1z,t_0+s_1 2\omega_n(z_0,z))+(z_0-s_1z,t_0-s_1 2\omega_n(z_0,z))\right)=\tfrac{1}{2}(2z_0,2t_0)=q.\nonumber
                \end{align}
                \end{proof}
 \begin{definition}
       Let $N$ be a homogeneous norm on $\mathbb{H}^n$. We say that $N$ has the \emph{geodesic linearity property} if every infinite geodesic is a horizontal line, that is,  if for every map $\gamma:\mathbb{R}\to\mathbb{H}^{n}$ with $d_{N}(\gamma(s_1),\gamma(s_2))=|s_1-s_2|$, for all $s_1,s_2\in\mathbb{R}$, there exists $z_0\in\mathbb{R}^{2n}\backslash\{0\}$ such that $\gamma(s)=\gamma(0)\ast (sz_0,0)$, for all $s\in\mathbb{R}$.
       \end{definition}

 \begin{ex}
 Let $d_{SR}$ be the standard sub-Riemannian distance on $\mathbb{H}^n$, that is, the sub-Finsler distance generated by the Euclidean norm $\|\cdot\|_2 = \sqrt{\langle \cdot,\cdot\rangle}$. The space $(\mathbb{H}^n,d_{SR})$ has the geodesic linearity property, but there exist finite geodesics which are not horizontal line segments with respect to $d_{SR}$, see for instance \cite{MR3417082}. This is different from the situation in normed spaces. As explained below Proposition \ref{p:strict_convex_normed} the geodesic linearity property of a normed space is equivalent to the fact that all geodesics are linear, not only the infinite ones.
 \end{ex}

       \begin{remark}
       The last definition is equivalent to ``every geodesic in $(\mathbb{H}^n, d_N)$ is a straight line'' since from Proposition \ref{p:horiz_line}  we know that a straight line can be reparameterized to be a geodesic if and only if it is horizontal.
       \end{remark}
In the following, we will see that the geodesic linearity property is implied by the previously discussed properties. The proof is basically an application of the next lemma, which appears for instance in
 \cite[(2)]{vai:mazur-ulam}.
\begin{lemma}\label{l:linearity}
Let $(V,\|\cdot \|_V)$ and $(W,\|\cdot \|_W)$ be two real normed spaces and $g:(V,\|\cdot \|_V)\to(W,\|\cdot \|_W)$ a map fulfilling
\begin{align}
&i)\ g(0)=0,\nonumber\\
&ii)\ g\left(\frac{v_1+v_2}{2}\right)=\frac{g(v_1)+g(v_2)}{2},\ \forall v_1,v_2\in V,\nonumber\\
&iii)\ \text{g is continuous}.\nonumber
\end{align}
Then, $g$ is linear.
\end{lemma}

       \begin{proposition}
Let $N$ be a homogeneous norm on $\mathbb{H}^n$ having the midpoint property. Then, $N$ has the geodesic linearity property.
\end{proposition}

\begin{proof}
Let $\gamma:\mathbb{R}\to\mathbb{H}^{n}$ be a map with $d_{N}(\gamma(s_1),\gamma(s_2))=|s_1-s_2|$, for all
$s_1,s_2\in\mathbb{R}$. Without loss of generality we can assume that $\gamma(0)=0$ (otherwise consider $\hat{\gamma}:=(\gamma(0))^{-1}\ast \gamma$).

\medskip

\noindent\textbf{Claim.}
$\gamma$ is $\mathbb{R}$-linear.

\begin{proof}[Proof of Claim]
The map $\gamma:(\mathbb{R},|\cdot |)\to(\mathbb{H}^{n},d_N)$ is clearly continuous with respect to the topology induced by $d_N$. Since this topology is equal the Euclidean topology on $\mathbb{R}^{2n+1}$, $\gamma$ is also continuous viewed as a map $\gamma:(\mathbb{R},|\cdot|)\to(\mathbb{R}^{2n+1}, \|\cdot \|_2)$. Furthermore, since by assumption $\gamma(0)=0$, in order to prove that $\gamma$ is linear, according to Lemma \ref{l:linearity} it suffices to check that
\begin{align}
 \gamma\left(\frac{s_1+s_2}{2}\right)=\frac{\gamma(s_1)+\gamma(s_2)}{2},\ \forall s_1,s_2\in\mathbb{R}.
 \end{align}
  For this, consider $s_1,s_2\in\mathbb{R}$. Defining $\bar{s}:=\frac{s_1+s_2}{2}$, we get
\begin{align}
d_N (\gamma(s_1),\gamma(s_2))=|s_1-s_2|=2|s_1-\bar{s}|=2d_N (\gamma(s_1),\gamma(\bar{s})),\nonumber\\
d_N (\gamma(s_1),\gamma(s_2))=|s_1-s_2|=2|s_2-\bar{s}|=2d_N (\gamma(s_2),\gamma(\bar{s})).
\end{align}
Since by assumption $N$ has the midpoint property, this implies
\begin{align}
\gamma\left(\frac{s_1+s_2}{2}\right)=\gamma(\bar{s})=\frac{\gamma(s_1)+\gamma(s_2)}{2}.\nonumber
\end{align}
\end{proof}

 The fact that $\gamma:\mathbb{R}\to\mathbb{R}^{2n+1}$ is linear means that it is actually a straight line that goes through the origin. Furthermore, since $\gamma$ is in particular a geodesic, it must be a horizontal line, and therefor there exists $z\in\mathbb{R}^{2n}\backslash\{0\}$ such that $\gamma(s)=(sz,0)$, $s\in\mathbb{R}$.
\end{proof}

\subsubsection{Strict convexity of projected norms}\label{s:StrictConvexityNorm}

In this section we provide one more sufficient condition and one necessary condition for a homogeneous norm on $\mathbb{H}^n$ to have the geodesic linearity property. These conditions are derived from the relation between $N$ and its `projection' $\|\cdot\|$.
First we observe that
Proposition \ref{p:geodesics}  has the following immediate consequence.

\begin{proposition}\label{p:SuffGLP}
Assume that $N$ is a homogeneous distance on $\mathbb{H}^n$ and let $\|z\|:= N((z,0))$. Denote by $d_{SF}$ the sub-Finsler distance associated to $\|\cdot\|$. If $(\mathbb{H}^n,d_{SF})$ has the geodesic linearity property, so does $(\mathbb{H}^n,d_N)$.
\end{proposition}

Proposition \ref{p:SuffGLP} provides a sufficient condition for geodesic linearity of a homogeneous norm $N$ in terms of the projected norm $\|\cdot\|$. In the following we give a necessary condition.

\begin{proposition}\label{p:strictConvexityNecessary}
If $N$ is a homogeneous norm on $\mathbb{H}^n$ such that $\|\cdot\|$, defined by $\|z\|:= N((z,0))$, is not strictly convex, then $(\mathbb{H}^n,d_N)$ does not have the geodesic linearity property.
\end{proposition}

\begin{proof}
Since $\|\cdot\|$ is not strictly convex, by Proposition \ref{p:strict_convex_normed}, there exists an infinite geodesic $\gamma_I:\mathbb{R}\to (\mathbb{R}^{2n},\|\cdot\|)$ which is not a line. Note that, being geodesic, this curve is Lipschitz, and hence  differentiable almost everywhere as a map into the Euclidean space $\mathbb{R}^{2n}$. We can lift $\gamma_I$ to a horizontal Lipschitz curve in $\mathbb{H}^n$. More precisely, integrating the formula in \eqref{eq:horiz_comp}, we find a function $\gamma_{2n+1}:\mathbb{R}\to \mathbb{R}$ such that $\gamma=(\gamma_I,\gamma_{2n+1}):\mathbb{R}\to \mathbb{H}^n$ is a horizontal curve.
To see that $\gamma$ is Lipschitz with respect to $d_N$, it suffices to verify that it is Lipschitz with respect to the sub-Finsler distance $d_{SF}$ associated to $\|\cdot\|$, and this is immediate:
\begin{displaymath}
d_{SF}(\gamma(s),\gamma(s')) \leq \int_s^{s'} \|\dot{\gamma}_I(\xi)\|\;\mathrm{d}\xi = |s-s'|,\quad\text{for all }-\infty<s<s'<\infty.
\end{displaymath}
We claim that it is a geodesic with respect to $d_N$. Indeed, we have for all $s<s'$ that
\begin{align*}
d_N(\gamma(s),\gamma(s'))&\leq L_N(\gamma|_{[s,s']})= \int_s^{s'} \|\dot{\gamma}_I(\xi)\|\;\mathrm{d}\xi= L_{\|\cdot\|}((\gamma_I)|_{[s,s']})
= \|\gamma_I(s)-\gamma_I(s')\|\\
&\leq d_N(\gamma(s),\gamma(s')).
\end{align*}
Here we have used (in this order), the metric definition of length, Proposition \ref{p:length_homogeneous} and Example \ref{ex:LengthNormed}, the geodesic property of $\gamma_I$ with the characterization in Lemma \ref{l:geodChar} and Lemma \ref{l:normHorizVert} with the definition of $N$ and $\|\cdot\|$. It follows that
\begin{displaymath}
d_N(\gamma(s),\gamma(s'))=  \|\gamma_I(s)-\gamma_I(s')\| = |s-s'|,\quad\text{for all }-\infty<s<s'<\infty,
\end{displaymath}
and hence $\gamma:\mathbb{R} \to (\mathbb{H}^n,d_N)$ is a geodesic. Clearly it is not a line since its projection $\gamma_I$ to $\mathbb{R}^{2n}$ is not a line. This shows that $(\mathbb{H}^n,d_N)$ does not have the geodesic linearity property.
\end{proof}

In the first Heisenberg group $\mathbb{H}^1$, the classification of the geodesics with respect to a sub-Finsler distance associated to a norm $\|\cdot\|$ is related to the following \emph{isoperimetric problem} on the Minkowski plane $(\mathbb{R}^2,\|\cdot\|)$: given a number $A$ find a closed path through $0$ of minimal $\|\cdot\|$-length which encloses (Euclidean) area $A$. To describe the solution, we introduce the following notation for the closed unit ball and dual ball in $(\mathbb{R}^2,\|\cdot\|)$:
\begin{displaymath}
B:= \{z\in\mathbb{R}^2:\; \|z\|\leq 1\}\quad\text{and} B^{\circ}:=\{w:\; \langle w,z\rangle\leq 1:\; z\in B\}.
\end{displaymath}
The \emph{isoperimetrix} $I$ is the boundary of $B^{\circ}$ rotated by $\pi/2$, and it can be parameterized as a closed curve.
Buseman \cite{10.2307/2371807} has proved that the solution to the above stated isoperimetric problem is given by (appropriate dilation and translation) of the isoperimetrix.  Note that if $\|\cdot\|$ is strictly convex, then $I$ is of class $\mathcal{C}^1$.
Based on Buseman's work and its interpretation in the Heisenberg context, one arrives at the following conclusion.

\begin{cor}\label{c:plane}
Let $N$ be a homogeneous norm on $\mathbb{H}^1$. Then $(\mathbb{H}^1,d_N)$ has the geodesic linearity property if and only if the norm defined on $\mathbb{R}^2$ by $\|z\|:=N((z,0))$ is strictly convex.
\end{cor}

\begin{proof}
Proposition \ref{p:strictConvexityNecessary} says that the geodesic linearity property of $d_N$ implies strict convexity of $\|\cdot\|$. For the reverse implication it suffices, according to Proposition \ref{p:SuffGLP}, to show that strict convexity of $\|\cdot\|$ implies the geodesic linearity property of the associated sub-Finsler distance. So let $d_{SF}$ be the sub-Finsler distance on $\mathbb{H}^1$ given by the norm $\|\cdot\|$. By left-invariance it is enough to show that all infinite geodesics in $(\mathbb{H}^1,d_{SF})$ which pass through the origin are straight lines. By \cite[Theorem 1]{MR1290095}, \cite[\S 4]{MR2586628} it is known that if $\|\cdot\|$ is strictly convex, then the geodesics  in $(\mathbb{H}^1,d_{SF})$ passing through $0$ project to the $(x,y)$-plane either to (i) straight lines or line segments, or (ii) isoperimetric paths passing through zero, see also Section 2.3 in \cite{zbMATH06363213}. By an \emph{isoperimetric path} we mean a subpath of a dilated and left-translated isoperimetrix in the sense of Buseman. Conversely, every horizontal lift of such a line segment or isoperimetric path through $0$ yields a geodesic in $(\mathbb{H}^1,d_{SF})$ passing through the origin.

Let $I$ be a (translated and dilated) isoperimetrix passing through $0$. This is a closed curve which can be lifted to a geodesic, say $\lambda: [0,\ell]\to (\mathbb{H}^1,d_{SF})$. We claim that $\lambda$ cannot be extended to a length minimizing curve on any larger interval, and thus stops to be an isometric embedding. The reason for this is that we can translate $I$ so that some other point passes through $0$ with a tangent different from the one of the original curve $I$ at $0$. Lifting the resulting curve, we obtain two different geodesics connecting the two points $0=\lambda(0)$ and $\lambda(\ell)$ on the $t$-axis. If we could extend to a length minimizing curve past the point $\lambda(\ell)$, we would construct by concatenation a geodesic segment containing $\lambda(0)$ which does not project to a isoperimetric path or a line segment. This is impossible and we see that the lifts of an isoperimetrix stop to be length minimizing after finite time. (See also the bottom of p.5 in \cite{MR1290095}.)

 It follows that the only infinite geodesics are horizontal lines, and the proof is complete.
\end{proof}

\section{The main result}\label{s:main}

In this section we study isometric embeddings of Euclidean spaces or Heisenberg groups into Heisenberg groups, with homogenous distances in the respective groups. The existence of an isometric embedding $f:\mathbb{R}^m \to \mathbb{H}^n$ imposes restrictions on $m$ and $n$. Namely, it is known from \cite{Ambrosio2000,Magnani2004} that $\mathbb{H}^n$ is purely $m$-unrectifiable for $m>n$, hence $\mathcal{H}^m(f(A))=0$ for every Lipschitz map $f:A \subseteq \mathbb{R}^m \to \mathbb{H}^n$ if $m>n$. (This holds for any choice of metric on $\mathbb{R}^m$ which is bi-Lipschitz equivalent to the Euclidean distance, and any choice of metric on $\mathbb{H}^n$ equivalent to the standard sub-Riemannian distance.) Since isometric embeddings are bi-Lipschitz mappings onto their domains and thus send positive $\mathcal{H}^m$-measure sets onto positive $\mathcal{H}^m$-measure sets, it follows that there does not exist an isometric embedding $f:\mathbb{R}^m \to \mathbb{H}^n$ if $m>n$. Moreover, there clearly cannot exist an isometric embedding $f:\mathbb{H}^m \to \mathbb{H}^n$ for $m>n$. Thus the range of parameters $m$ and $n$ in the Theorem \ref{t:main_Eucl_Heis} below is the natural one.

\begin{thm}\label{t:main_Eucl_Heis}
Let $\mathbb{G}_1 \in \{(\mathbb{R}^m,+),(\mathbb{H}^m,\ast)\}$ and $\mathbb{G}_2 = (\mathbb{H}^n,\ast)$, $m\leq n$, be endowed with left invariant-homogeneous distances $d_1$ and $d_2$, respectively. If $d_2$ satisfies the geodesic linearity property, then every isometric embedding $f:(\mathbb{G}_1,d_1) \to (\mathbb{G}_2,d_2)$ is of the form $f=L_p \circ A$, where $L_g$ denotes left translation by an element $g\in \mathbb{G}_2$ and $A: \mathbb{G}_1 \to \mathbb{G}_2$ is a homogeneous homomorphism.
\end{thm}

\begin{proof}
We first prove the theorem in the case $\mathbb{G}_1= \mathbb{R}^m$ and $\mathbb{G}_2= \mathbb{H}^n$. Assume that $\mathbb{R}^m$ is endowed with a norm $\|\cdot\|$ and $\mathbb{H}^n$ is equipped with a homogeneous norm $N$ and associated homogeneous distance $d_N$. By post-composing with a left translation if necessary, we may assume without loss of generality that $f(0)=0$ and we will show that $f$ equals a homogeneous homomorphism $A$.

Every line (affine $1$-dimensional space) in $\mathbb{R}^m$ can be parameterized as an infinite geodesic $\ell: \mathbb{R} \to (\mathbb{R}^m,\|\cdot\|)$. Since $f$ is an isometric embedding, $f \circ \ell$ is an infinite geodesic in $(\mathbb{H}^n,d_N)$, and thus, by the assumption on the geodesic linearity property and Proposition \ref{p:horiz_line}, a horizontal line.

Now every point in $\mathbb{R}^m$ lies on a line through the origin, whose image must be a horizontal line in through the origin in $\mathbb{R}^{2n}\times \{0\} \subset \mathbb{H}^n$ by what we said above. It follows that $f(\mathbb{R}^m) \subseteq \mathbb{R}^{2n} \times \{0\}$ and $f$ is of the form
\begin{displaymath}
f(x) = (T(x),0),\quad\text{for all }x\in\mathbb{R}^m,
\end{displaymath}
for a suitable mapping $T:\mathbb{R}^m \to \mathbb{R}^{2n}$. We will show that $T$ is linear, thus proving the claim that $f$ is a homogeneous homomorphism. To see this, consider arbitrary $x,y\in\mathbb{R}^m$ and $s\in \mathbb{R}$. By what we said so far and since $f$ is isometric, in particular along lines, we know that there exist $z,z_0$, and $\zeta$ in $\mathbb{R}^{2n}$ such that
\begin{displaymath}
f( ys )= (z s,0)\quad\text{and}\quad f(x+ys)= (z_0 + \zeta s,0),\text{for all }s\in \mathbb{R}.
\end{displaymath}
Using the fact that $f$ preserves distances, we find
\begin{align*}
\|x\|&= \|(x+ys)-ys\| = d_N((z_0+\zeta s,0),( zs ,0))= N ((z_0 + s (\zeta-z),2 \omega_n(-zs,z_0 + \zeta s))) \\&\gtrsim |s| \|\zeta-z\|_2 -\|z_0\|_2,
\end{align*}
where the last inequality (which holds up to an absolute multiplicative constant) follows from the equivalence of homogeneous norms on $\mathbb{H}^n$. Letting $|s|$ tend to infinity, we arrive at a contradiction unless $\|\zeta-z\|_2=0$. Thus we conclude that necessarily $\zeta = z$. Using this information and an analogous argument as before, we deduce further that
\begin{displaymath}
\|x\| \gtrsim \sqrt{|\omega_n(-zs,z_0 + \zeta s)|}= \sqrt{|\omega_n(-zs,z_0 + \zeta s)|} = \sqrt{|s|} \sqrt{|\omega_n(-z,z_0)|}.
\end{displaymath}
Letting $|s|$ tend to infinity, we conclude that necessarily $\omega_n(z,z_0)=0$.

We have therefore that
\begin{equation}\label{eq:additivity}
(T(x+y),0)=f\left(x+y\right)=(z_0 +\zeta,0)=(z_0 + z,0)= (Tx+ Ty,0)
\end{equation}
and
\begin{displaymath}
f(x+y)=f(x)\ast f(y),
\end{displaymath}
which shows that $f$ is a group homomorphism. It remains to verify the homogeneity.

From \eqref{eq:additivity} we deduce that
\begin{displaymath}
T(x+y)= T(x)+ T(y), \quad \text{for all }x,y\in\mathbb{R}^{m}.
\end{displaymath}
From this identity it follows further that $T(x)/2 = T(x/2)$ for all $x\in \mathbb{R}^m$. Thus we conclude that $T:\mathbb{R}^m \to \mathbb{R}^{2n}$ is a continuous map with $T(0)=0$ and the following property holds
\begin{displaymath}
T\left(\frac{x+y}{2}\right) = \frac{T(x)+T(y)}{2},\quad \text{for all }x,y\in\mathbb{R}^m.
\end{displaymath}
Lemma \ref{l:linearity} implies that $T$ is linear. This concludes the first part in the proof of Theorem \ref{t:main_Eucl_Heis}.

\medskip

Next we prove the theorem in the case $\mathbb{G}_1= \mathbb{H}^m$, $\mathbb{G}_2= \mathbb{H}^n$, with left-invariant metrics $d_1$ and $d_2$ induced by homogeneous norms $N_1$ and $N_2$. Since left translations are isometries, we may again assume without loss of generality that $f(0)=0$, and it suffices to show that $f$ is a homogeneous homomorphism $A$. Every horizontal line in $\mathbb{H}^m$ can be parameterized as a geodesic $\ell:\mathbb{R} \to (\mathbb{H}^m,d_1)$. Since $f$ is an isometric embedding, $f\circ \ell$ is an infinite geodesic in $(\mathbb{H}^n,d_2)$ and thus, by the geodesic linearity property, a horizontal line.
Hence, for every $z\in \mathbb{R}^{2m}$ and $p_0\in \mathbb{H}^m$, there exist $T(z,p_0) \in \mathbb{R}^{2n}$ and $q_0(z,p_0)\in \mathbb{H}^n$ such that
\begin{equation}\label{eq:imageLine}
f(p_0 \ast (sz,0)) = q_0(z,p_0) \ast (s T(z,p_0),0),\quad \text{for all }s\in \mathbb{R}.
\end{equation}
We will show that $q_0$ depends only on $p_0$ and that $T$ depends only on $z$.

Since $f$ is isometric on the line $s\mapsto p_0 \ast (sz,0)$, we have by left-invariance of the norms that $N_1((z,0))=N_2((T(z,p_0),0))$.
Inserting $s=0$ in the formula in \eqref{eq:imageLine}, we find $f(p_0)=q_0(z,p_0)$ and hence
\begin{equation}\label{eq:first_form_f}
f(p_0 \ast (sz,0)) = f(p_0) \ast (s T(z,p_0),0),\quad\text{for all }s\in \mathbb{R}.
\end{equation}
We show that $T$ depends only on $z$, but not on $p_0$, in other words, every fibration determined by a left-invariant horizontal vector field in $\mathbb{H}^m$ is mapped onto an analogous fibration in the target. To see this, we exploit the comparability of all homogeneous distances on a Heisenberg group. For convenience, we denote $f(p)=(\zeta(p),\tau(p)) \in \mathbb{R}^{2n}\times \mathbb{R}$. This yields by \eqref{eq:first_form_f} for all $z\in\mathbb{R}^{2m}$, $p_0=(z_0,t_0) \in \mathbb{H}^m$, and $s\in\mathbb{R}$ that
\begin{align*}
\|T(z,p_0)-T(z,0)\|_2 |s|- \|\zeta(p_0)\|_2 & \leq \|\zeta(p_0) + s(T(z,p_0)-T(z,0))\|_2\\
&\lesssim d_2(f(p_0 \ast (sz,0)),f((sz,0)))\\
&= d_1 ((p_0 \ast (sz,0)),(sz,0))\\
&\lesssim \|z_0\|_2 + \sqrt{|t_0 + 4 s \omega_m(z_0,z)|}\\
&\leq \|z_0\|_2 + \sqrt{|t_0|} + 2 \sqrt{|s|}\sqrt{|\omega_m(z_0,z)|}.
\end{align*}
We observe that the left-hand side of the above chain of inequalities grows linearly as $|s|\to \infty$, whereas the right-hand side exhibits only a sub-linear growth. This would lead to a contraction, unless
\begin{displaymath}
T(z,p_0)=T(z,0)
\end{displaymath}
which must hence be the case. Thus we have found that
\begin{equation}\label{eq:first_formula_f}
f(p_0 \ast (sz,0)) = f(p_0) \ast (sT(z),0),\quad \text{for all }s\in\mathbb{R},\,p_0\in\mathbb{H}^m,\,z\in\mathbb{R}^{2m}
\end{equation}
for a suitable function $T:\mathbb{R}^{2m}\to \mathbb{R}^{2n}$. In particular, by choosing $p_0=(0,t)$ and $s=1$, we find
\begin{equation}\label{eq:sep_vertical}
f((z,t)) = f((0,t))\ast (T(z),0),\quad\text{for all }(z,t)\in \mathbb{R}^{2m}\times \mathbb{R},
\end{equation}
and by choosing $p_0=0$, we see that
\begin{equation}\label{eq:f_horiz}
 (T(sz),0)=f((sz,0))= (sT(z),0),\quad\text{for all }z\in \mathbb{R}^{2m},\,s\in\mathbb{R}.
\end{equation}

In the next step we will show that there exists a function $h:\mathbb{R}\to \mathbb{R}$ such that $f(0,t)=(0,h(t))$ for all $t\in\mathbb{R}$, that is, the vertical axis gets mapped to the vertical axis. To see this we use the fact that for arbitrary $t\in\mathbb{R}$, the two points $(0,0)$ and $(0,t)$ can be connected by a concatenation of four suitable horizontal line segments, parameterized as follows:
\begin{align*}
\ell_1(s)&:= \left( -s \tfrac{t}{4} e_1,0\right), & s\in [0,1]\\
\ell_2(s)&:= \left(- \tfrac{t}{4} e_1,0\right) \ast \left( s e_{m+1},0\right),& s\in [0,1]\\
\ell_3(s)&:= \left(-\tfrac{t}{4}e_1 + e_{m+1},\tfrac{t}{2}\right)\ast \left( s \tfrac{t}{4}e_1,0\right),&s\in [0,1]\\
\ell_4(s)&:= \left(e_{m+1},t\right)\ast \left(-s e_{m+1},0\right),&s\in [0,1].
\end{align*}
Here $e_1$ and $e_{(m+1)}$ denote the standard first (respectively $(m+1)$-th) standard unit vector in $\mathbb{R}^{2m}$. In particular, we can write
\begin{displaymath}
(0,t) = \left(-\tfrac{t}{4}e_1,0\right)\ast \left(e_{m+1},0\right) \ast \left(\tfrac{t}{4}e_1,0\right) \ast \left(-e_{m+1},0\right).
\end{displaymath}
We apply $f$ to both sides of the equation and apply iteratively the identity \eqref{eq:first_formula_f}. In this way we obtain two different formulae for the same point in $\mathbb{H}^n$. By comparing the projection to $\mathbb{R}^{2n}\times \{0\}$, we find
\begin{displaymath}
\zeta(0,t)= T\left(-\tfrac{t}{4}e_1\right) + T\left(e_{m+1}\right) + T\left(\tfrac{t}{4}e_1\right) + T\left(-e_{m+1}\right),
\end{displaymath}
which, by the homogeneity of $T$ established in \eqref{eq:f_horiz}, yields $\zeta(0,t)=0$. Using \eqref{eq:sep_vertical}, we conclude that $f$ is of the form
\begin{displaymath}
f(z,t)=(\zeta(z,t),\tau(z,t))= (T(z),h(t)),\quad \text{for all }(z,t)\in \mathbb{R}^{2m}\times \mathbb{R}
\end{displaymath}
with $T( sz)=s T(z)$ for all $s\in \mathbb{R}$.
The restriction of the map $f$ to the vertical axis maps the vertical axis in $\mathbb{H}^m$ to the vertical axis in $\mathbb{H}^n$, and since it is an isometric embedding, this mapping must in fact be surjective. The identity
\begin{displaymath}
N_1((0,t))= N_2((0,h(t))),\quad\text{for all }t\in \mathbb{R},
\end{displaymath}
then implies that $h(t)= a t$ for a suitable constant $a\in\mathbb{R}$.

We plug this formula into the identity \eqref{eq:sep_vertical}. This yields for all $z,z_0\in\mathbb{R}^{2m}$ and $t_0\in \mathbb{R}$ that
\begin{align*}
(T(z_0+z),a t_0 + 2 a\omega_n(z_0,z))&= f((z_0+z,t_0 + 2\omega_m(z_0,z))\\
&= f((z_0,t_0))\ast (T(z),0)\\
&= (T(z_0),a t_0)\ast (T(z),0)\\
&= (T(z_0)+T(z),a t_0 + 2\omega_n(T(z_0),T(z))).
\end{align*}
Hence we conclude for all $z,z_0 \in \mathbb{R}^{2m}$ that
\begin{displaymath}
\left\{\begin{array}{l}T(z_0+z)=T(z_0) + T(z)\\ a \omega_n(z_0,z) = \omega_n(T(z_0),T(z)).\end{array} \right.
\end{displaymath}
The first condition shows together with the homogeneity established in \eqref{eq:f_horiz} that $T:\mathbb{R}^{2m}\to \mathbb{R}^{2n}$ is linear. The characterization of homogeneous homomorphisms given in Lemma \ref{l:homohHomo} concludes the proof of the theorem.
\end{proof}

\section{Examples of homogeneous norms}\label{ex:Examples}

\subsection{Norms}\label{s:exNorms}

We give a few examples of homogeneous norms on $\mathbb{H}^n$, both classical and new ones, and prove their properties regarding convexity.
More examples of homogeneous norms on $\mathbb{H}^n$ can be found for instance in \cite{MR1067309,Lee06lpmetrics,LDNG}.


\begin{ex}[Kor\'{a}nyi-Cygan norm]\label{t:KoranyiNorm}

Let $n\in\mathbb{N}_{\geq1}$ and $\|\cdot \|_2$ be the Euclidean norm on $\mathbb{R}^{2n}$. Then, the map
\begin{align}
 N_{K}:\mathbb{H}^{n}\to\mathbb{R}_{\geq 0},\ (z,t)\mapsto((\|z\|_2)^{4}+t^{2})^{\frac{1}{4}},\nonumber
 \end{align}
 defines a horizontally strictly convex homogenous norm on the Heisenberg group.
\end{ex}

This is one of the best known homogeneous norms on the Heisenberg group, partially because of its role in the definition of the fundamental solution of the sub-Laplacian on $\mathbb{H}^n$ found by G.\ B.\ Folland in \cite{MR0315267}.

\begin{proof}
The fact that $N_K$ defines a homogeneous norm is well known and was first proved by J.\ Cygan in \cite{zbMATH03745826}. We include here a proof for the triangle inequality because this is needed in establishing the horizontal strict convexity of $N_K$. For this, consider $(z,t),(z',t')$ in $\mathbb{H}^{n}$. In this case, it is convenient to use the complex notation of the Heisenberg group and interpret $z$ and $z'$ as vectors in $\mathbb{C}^{n}$ (this is possible if we identify the elements of these two spaces in the way we did in Section \ref{ss:TheHeisenberggroup}). Taking into account the Cauchy-Schwarz inequality, we get
\begin{align}
 \left(N_{K}((z,t\right)*(z',t')))^{2}&=\left(\left(\sum_{j=1}^{n}|z_{j}+z_{j}'|^{2}\right)^{2}+\left(t+t'+2\sum_{j=1}^{n}\mathrm{Im}(z_{j}\bar{{z'_{j}}})\right)^{2}\right)^{\frac{1}{2}}\nonumber\\
 &=\left|\left(\sum_{j=1}^{n}(|z_{j}|^{2}+|z'_{j}|^{2}+2\mathrm{Re}(z_{j}\bar{{z'_{j}}})\right)+\mathrm{i}\left(t+t'+2\sum_{j=1}^{n}\mathrm{Im}(z_{j}\bar{{z'_{j}}})\right)\right|\nonumber\\
 &=\left|(\sum_{j=1}^{n}|z_{j}|^{2})+\mathrm{i}t+(\sum_{j=1}^{n}|z'_{j}|^{2})+\mathrm{i}t'+2\sum_{j=1}^{n}z_{j}\bar{{z'_{j}}}\right|\nonumber\\
 &\leq\left|\sum_{j=1}^{n}|z_{j}|^{2}+\mathrm{i}t\right|+\left|\sum_{j=1}^{n}|z_{j}'|^{2}+\mathrm{i}t'\right|+2\left|\sum_{j=1}^{n}z_{j}\bar{z_{j}'}\right|\label{ineq:1}\\
 &\leq\left|\sum_{j=1}^{n}|z_{j}|^{2}+\mathrm{i}t\right|+\left|\sum_{j=1}^{n}|z_{j}'|^{2}+\mathrm{i}t\right|+2\sum_{j=1}^{n}|z_{j}||z_{j}'|\label{ineq:2}\\
&=(N_{K}(z,t))^{2}+(N_{K}(z',t'))^{2}+2\sum_{j=1}^{n}|z_{j}||z_{j}'|\nonumber\\
&\leq(N_{K}(z,t))^{2}+(N_{K}(z',t'))^{2}+2(\sum_{j=1}^{n}|z_{j}|^{2})^{\frac{1}{2}}(\sum_{j=1}^{n}|z'_{j}|^{2})^{\frac{1}{2}}\label{ineq:3}\\
&\leq(N_{K}(z,t))^{2}+(N_{K}(z',t'))^{2}+2N_{K}(z,t))N_{K}(z',t')\label{ineq:4}\\
&=(N_{K}(z,t)+N_{K}(z',t'))^{2}.\nonumber
\end{align}
Now, for proving horizontal strict convexity, assume that $p=(z,t)\neq(0,0)\neq(z',t')=p'$ and that $N_K(p\ast p') = N_K(p) +N_K(p')$.
Then equality must hold in \eqref{ineq:1}, \eqref{ineq:2}, \eqref{ineq:3} and \eqref{ineq:4}.
First, \eqref{ineq:4} implies that
\begin{displaymath}
\left((\sum_{j=1}^{n}|z_{j}|^{2})^{2}\right)^{\frac{1}{4}}\left((\sum_{j=1}^{n}|z'_{j}|^{2})^{2}\right)^{\frac{1}{4}}=\left((\sum_{j=1}^{n}|z_{j}|^{2})^{2}+t^{2}\right)^{\frac{1}{4}}\left((\sum_{j=1}^{n}|z'_{j}|^{2})^{2}+t'^{2}\right)^{\frac{1}{4}},
\end{displaymath}
from which we conclude that $t=t'=0$, $z\neq 0$ and $z'\neq 0$.

On the other hand, \eqref{ineq:2} and \eqref{ineq:3} yield
\begin{displaymath}
\left|\sum_{j=1}^{n}z_{j}\bar{z_{j}'}\right|=\left(\sum_{j=1}^{n}|z_{j}|^{2}\right)^{\frac{1}{2}}\left(\sum_{j=1}^{n}|z'_{j}|^{2}\right)^{\frac{1}{2}}.
\end{displaymath}
Since in the Cauchy-Schwarz inequality on $\mathbb{C}^{n}$ equality holds only if the involved vectors are linearly dependent, it follows that there must exist a complex number $\alpha\neq0$ such that $z'=\alpha z$. Furthermore, \eqref{ineq:1} with $t=t'=0$ is equivalent to
\begin{equation}\label{eq:conseqEqComplex}
\left|(\sum_{j=1}^{n}|z_{j}|^{2})+(\sum_{j=1}^{n}|z'_{j}|^{2})+2\sum_{j=1}^{n}z_{j}\bar{{z'_{j}}}\right|=\sum_{j=1}^{n}|z_{j}|^{2}+\sum_{j=1}^{n}|z_{j}'|^{2}+2\sum_{j=1}^{n}|z_{j}||z_{j}'|.
\end{equation}
Inserting $z'=\alpha z$ in \eqref{eq:conseqEqComplex}, and using the fact that $|w+w'|=|w|+|w'|$ if and only if $w\bar{w'}\in\mathbb{R}_{\geq0},\ \forall w,w'\in\mathbb{C}$, we obtain
\begin{align*}
&|(\|z\|_2)^2+|\alpha|^2(\|z\|_2)^2+2\bar{\alpha}(\|z\|_2)^2|=(\|z\|_2)^2+|\alpha|^2 (\|z\|_2)^2+2|\alpha|(\|z\|_2)^2\nonumber\\
&\Leftrightarrow |1+|\alpha|^2+2\bar{\alpha}|=|1+|\alpha|^2|+|2\bar{\alpha}|,\ \Leftrightarrow(1+|\alpha|^2)2\alpha\in\mathbb{R}_{\geq0},\ \Leftrightarrow\alpha\in\mathbb{R}_{>0}.
\end{align*}
As result, we obtain that $p=(z,0)$ and $p'=(\alpha z,0)$ with $\alpha\in\mathbb{R}$, which proves the horizontal strict convexity of $N_K$.

\end{proof}

\begin{ex}[Lee-Naor norm]\label{t:LeeNaor}
Let $N_K$ be the Koranyi-Cygan norm on $\mathbb{H}^{n}$, and $\|\cdot \|_2$ be the Euclidean norm on $\mathbb{R}^{2n}$. Then, the map
\begin{align}
 N:\mathbb{H}^{n}\to\mathbb{R}_{\geq0},\ (z,t)\mapsto\sqrt{(N_{K}(z,t))^{2}+(||z||_2)^{2}},\nonumber
 \end{align}
 defines a horizontally strictly convex homogenous norm on the Heisenberg group.
\end{ex}

The norm in Example \ref{t:LeeNaor} has appeared independently in different contexts. J.\ Lee and A.\ Naor \cite{Lee06lpmetrics}  showed that $\sqrt{d_N}$ is a metric of negative type on $\mathbb{H}^1$, that is, $(\mathbb{H}^1,\sqrt{d_N})$ admits an isometric embedding into Hilbert space.
This provided a counterexample to the so-called Goemans-Linial conjecture, since it follows at the same time by the work of
J.\ Cheeger and B.\ Kleiner that $(\mathbb{H}^1,d_N)$ does not biLipschitzly embed into $L^1$. The distance $d_N$ was also used by Le Donne and Rigot in \cite{LDR} as an example of a homogeneous distance on $\mathbb{H}^n$ for which the Besicovitch covering property holds. In fact, $d_N$ is a particular instance of a whole family of homogeneous norms which were constructed by W.\ Hebisch and A.\ Sikora \cite{MR1067309}, and for which Le Donne and Rigot established the Besicovitch covering property.

\begin{proof} It is known that $N$ defines a homogeneous norm, see \cite{Lee06lpmetrics, LDR}.
As in the proof of Theorem \ref{t:KoranyiNorm}, the horizontal strict convexity will be deduced from a careful inspection of the proof of the triangle inequality. For this, let $(z,t),\ (z',t')\in\mathbb{H}^{n}$. First, we remark that
\begin{displaymath}
N_{K}(z,t)N_{K}(z',t')+\|z\|_{2}\|z'\|_{2}\leq\left(N_{K}(z,t)^2+\|z\|_{2}^2\right)^{\frac{1}{2}} \left(N_{K}(z',t')^2+\|z'\|_{2}^2\right)^{\frac{1}{2}}.
\end{displaymath}
Using this, we obtain
\begin{align}
&N((z,t)\ast (z',t'))^2=N_{K}((z,t)\ast (z',t'))^2+\|z+z'\|_{2}^2\nonumber\\
&\leq(N_{K}(z,t)+N_{K}(z',t'))^2+(\|z\|_{2}+\|z'\|_{2})^2\label{ineq_Lee1}\\
&=N_{K}(z,t)^2+\|z\|_2^2+N_{K}(z',t')^2+\|z'\|_2^2+2\left(N_{K}(z,t)N_{K}(z',t')+\|z\|_{2}\|z'\|_{2}\right)\nonumber\\
&\leq N_{K}(z,t)^2+\|z\|_2^2+N_{K}(z',t')^2+\|z'\|_2^2+2\left(N_{K}(z,t)^2+\|z\|_{2}^2\right)^{\frac{1}{2}} \left(N_{K}(z',t'))^2
+\|z'\|_{2}\right)^2)^{\frac{1}{2}}\nonumber\\
&=\left(\left(N_{K}(z,t)^2+\|z\|_2)^2\right)^{\frac{1}{2}}+\left(N_{K}(z',t')^2+\|z'\|_2)^2\right)^{\frac{1}{2}}\right)^2\nonumber\\
&=\left(N(z,t)+N(z',t')\right)^2\nonumber.
\end{align}

If $N((z,t)\ast (z',t')) = N((z,t))+ N((z',t'))$, then equality must hold everywhere in the above chain of estimates.
In particular, we have by \eqref{ineq_Lee1} that $N_K((z,t)\ast(z',t'))=N_{K}(z,t)+N_{K}(z',t')$, which according to Theorem \ref{t:KoranyiNorm} implies that $p=(z,t)$ and $p'=(z',t')$ lie on a horizontal line through the origin, if both are nonzero. This means that $N$ is horizontally strictly convex.
\end{proof}
In Section \ref{ss:NotionsOfStrictConvexity} we saw that the midpoint property implies geodesic linearity. Now we present an example which shows that the converse does not hold in general. This example belongs to a whole family of homogeneous norms, constructed using not only the Euclidean norm but the entire spectrum of $p$-norms on $\mathbb{R}^{2n}$. The properties of these norms depend on the value of $p$. In particular, for studying these properties we will often use the exact value of the best Lipschitz constant between different $p$-norms, presented in the following lemma.

\begin{lemma}\label{l:pNorms}
Let $1\leq p<q\leq\infty$ and $\|\cdot \|_r$ be the $r$-norm on $\mathbb{R}^{n}$, $r\in\{p,q\}$. Then, for all $x\in\mathbb{R}^{n}$, it holds
\begin{align}
\|x\|_{q}\leq\|x\|_{p}\leq\|x\|_{q}n^{\frac{1}{p}-\frac{1}{q}}.\nonumber
\end{align}
\end{lemma}

The above lemma can be obtained from the H\"older inequality and elementary calculations.

\begin{ex}\label{t:Npa_norm}
Let $n\in\mathbb{N}$,  $p\in\left[1,\infty\right]$, and let $\|\cdot \|_{p}$ be the $p$-norm on $\mathbb{R}^{2n}$  and $a\in\left(0,\infty\right)$. Then the function
\begin{align}
N_{p,a}:\mathbb{H}^n\to\mathbb{R},\ (z,t)\mapsto \max\left\{ ||z||_{p},a\sqrt{|t|}\right\},\nonumber
\end{align}
defines a norm on $\mathbb{H}^n$, if and only if
\begin{align}
 &i)\ 1\le p\le 2\ \text{and}\ 0<a\le1,\nonumber\\
 &\text{or}\nonumber\\
 &ii)\ 2<p\le\infty\ \text{and}\ 0<a\le n^{1/p-1/2}.\nonumber
 \end{align}
 In both cases, $N_{p,a}$ is homogenous.
\end{ex}

Due to its simplicity, the norm $N_{2,1}$ has often been used in literature, see for instance \cite{S}. To the best of our knowledge, the norms $N_{p,a}$ for $p\neq 2$ have not been studied in detail before.

\begin{proof}
The only nontrivial assertion is the triangle inequality:
\begin{displaymath}
N_{p,a}((z,t)\ast(z',t'))\leq N_{p,a}(z,t)+N_{p,a}(z',t'),
\end{displaymath}
which is equivalent to
\begin{equation}\label{eq:equivalent to triangle}
\left\{\begin{array}{ll}\|z+z'\|_{p}&\leq N_{p,a}(z,t)+N_{p,a}(z',t')\\\text{and}& \\a\sqrt{|t+t'+2\langle z,J_nz'\rangle|}&\leq N_{p,a}(z,t)+N_{p,a}(z',t').\end{array}\right.
\end{equation}
From the triangle inequality for the $p$-norm $\|\cdot\|_{p}$ on $\mathbb{R}^{2n}$ and the definition of $N_{p,a}$, we see that the first condition in \eqref{eq:equivalent to triangle} is always fulfilled. Hence, $N_{p,a}$ defines a homogeneous norm if and only if the second condition in \eqref{eq:equivalent to triangle} is fulfilled for every $(z,t),(z',t')$ in $\mathbb{R}^{2n}\times\mathbb{R}$. First, assume that $1\leq p\leq2$: \newline\newline
If $0<a\leq1$, using Cauchy-Schwarz inequality and Lemma \ref{l:pNorms}, we get
\begin{align}
a^2|\langle z,J_nz'\rangle|\leq |\langle z,J_nz'\rangle|\leq\|z\|_2 \|J_nz'\|_2=\|z\|_2 \|z'\|_2\leq \|z\|_p \|z'\|_p.
\end{align}
This implies for all $(z,t),(z',t')\in\mathbb{H}^{n}$ that
\begin{align}
&\left(a\sqrt{|t+t'+2\langle z,J_nz'\rangle|}\right)^{2}\leq a^{2}|t|+a^2|t'|+2a^2 |\langle z,J_nz'\rangle|\nonumber\\
&\leq a^{2}|t|+a^2|t'|+2\|z\|_{p}\|z'\|_{p}\label{eq:estimate_for_triangle}\\
&\leq \max\left\{\|z\|_{p},a\sqrt{|t|}\right\}^{2}+\max\left\{\|z'\|_{p},a\sqrt{|t'|}\right\}^{2}+2\max\left\{\|z\|_{p},a\sqrt{|t|}\right\}\max\left\{ \|z'\|_{p},a\sqrt{|t'|}\right\}\nonumber\\
&=\left(N_{p,a}(z,t)+N_{p,a}(z',t')\right)^{2}.\nonumber
\end{align}
Hence \eqref{eq:equivalent to triangle} and the triangle inequality hold.
On the other hand, if $a>1$, choosing $z:=e_1$, $z':=-e_{n+1}$, $t:=1/a^2$, and $t':=1/a^2$, we have
\begin{align}
&a\sqrt{|t+t'+2\langle z,J_n z'\rangle|}=\sqrt{2+2a^2}>2=N_{p,a}(z,t)+N_{p,a}(z',t'),\nonumber
\end{align}
and thus \eqref{eq:equivalent to triangle} and the triangle inequality fail.

Now, assume $2<p\leq\infty$.
If $0<a\leq n^{\frac{1}{p}-\frac{1}{2}}$, using again Cauchy-Schwarz inequality and Lemma \ref{l:pNorms}, we get
\begin{align}\label{eq:OtherRangep}
a^2|\langle z,J_nz'\rangle|&\leq a^2\|z\|_2 \|z'\|_2\leq a^2n^{1/2-1/p}\|z\|_p n^{1/2-1/p}\|z'\|_p  \leq\|z\|_{p}\|z'\|_{p}.
\end{align}
By the computation as in \eqref{eq:estimate_for_triangle},  this implies the triangle inequality. Finally, if $n^{\frac{1}{p}-\frac{1}{2}}<a<\infty$, taking $z:=\sum_{j=1}^{n}e_{j}$, $z':=-\sum_{j=n+1}^{2n}e_{j}$, $t:=\frac{n^{\frac{2}{p}}}{a^{2}}$ and $t':=\frac{n^{\frac{2}{p}}}{a^{2}}$, we obtain
\begin{align}
&a\sqrt{|t+t'+2 \langle z,J_nz'\rangle|}=\sqrt{{2n^{\frac{2}{p}}}+2na^2}>2n^{\frac{1}{p}}=N_{p,a}(z,t)+N_{p,a}(z',t').\nonumber
\end{align}
\end{proof}

In the following we show that for a specific choice of parameters $p$ and $a$,
Example \ref{t:Npa_norm} proves that the geodesic linearity property is not equivalent to the midpoint property.

\begin{proposition}\label{p:N_pa}
  Let $p\in[1,\infty]$ and  $a>0$ be such that the function $N_{p,a}$ (defined as in Example \ref{t:Npa_norm}) is a norm on $\mathbb{H}^n$. Then, $N_{p,a}$ does not have the midpoint property.
    \end{proposition}
  \begin{proof}
  Choosing $\hat{p}:=(e_1,0)$, and $q=(0,1/a^2)$, it holds
  \begin{align}
  &d_{N_{p,a}}(\hat{p},\hat{p}^{-1})= \max\lbrace\|-2e_1\|_p,a\sqrt{|0|}\rbrace=2,\nonumber\\
  &d_{N_{p,a}}(\hat{p},q)=\max\lbrace\|-e_1\|_p,a\sqrt{|1/a^2|}\rbrace=1,\nonumber\\
  &d_{N_{p,a}}(\hat{p}^{-1},q)=\max\lbrace\|e_1\|_p,a\sqrt{|1/a^2|}\rbrace=1.\nonumber
   \end{align}
   This means that  $d_{N_{p,a}}(\hat{p},\hat{p}^{-1})=2d_{N_{p,a}}(\hat{p},q)=2d_{N_{p,a}}(\hat{p}^{-1},q)$, but $q\neq \frac{\hat{p}+\hat{p}^{-1}}{2}$.
   \end{proof}

   We now obtain examples of a geodesic linear norm without midpoint property.

      \begin{thm}\label{t:Npa_GLP}
Let $n\in\mathbb{N}$, $p\in[1,\infty]$ and $a>0$ such that the map $N_{p,a}$ (defined as in Example \ref{t:Npa_norm}) is a homogeneous norm on $\mathbb{H}^{n}$. Then, $N_{p,a}$ has the geodesic linearity property if and only if $p\in(1,\infty)$. Moreover, in this case every finite geodesic is a horizontal line segment.
\end{thm}
\begin{proof}
Since the norm on $\mathbb{R}^{2n}$ defined through $z\mapsto N_{p,a}(z,0)$ is nothing else than the $p$-norm, we remark that for $p\in\{1,\infty\}$, the norm $N_{p,a}$ cannot have the geodesic linearity property since according to Proposition \ref{p:strictConvexityNecessary} this would require the strict convexity of $\|\cdot \|_p$. For $n=1$, Corollary \ref{c:plane} allows us to conclude in converse direction that
$N_{p,a}$  has the geodesic linearity property for $p\in(1,\infty)$ since, in this case, $\|\cdot\|_p$ is a strictly convex norm on $\mathbb{R}^2$.

For $n>1$ we verify the geodesic linearity property by explicit estimations.
So let $p\in(1,\infty)$ and $\gamma:([0,1],|\cdot |)\to(\mathbb{H}^{n},d_{N_{p,a}})$ be a geodesic with $\gamma(0)=0$. We need to show that for an appropriate $z_0\in\mathbb{R}^{2n}$, $\gamma$ can be written as $\gamma(s)=(sz_0,0),\ s\in[0,1]$. We can write $\gamma(s)=(z(s),t(s))$, with continuous functions $z:[0,1]\to\mathbb{R}^{2n}$ and $t:[0,1]\to\mathbb{R}$, such that $z(0)=0$ and $t(0)=0$. The proof is a succession of steps formulated as claims.

\medskip\noindent{\textbf{Claim 1.}}
 $||z(s)||_p\geq a\sqrt{|t(s)|}$,  for all $s\in(0,1)$.

\begin{proof}
Assume by contradiction that  $\|z(s_0)\|_p< a\sqrt{|t(s_0)|}$ for some $s_0\in(0,1)$. By continuity, there exists an interval $[b,c]\subseteq(0,1)$ such that $s_0\in[b,c]$ and $\|z(s)\|_p<a\sqrt{|t(s)|}$, for all $s\in[b,c]$. Defining  $g_1:=\gamma(b)$ and $g_2:=(\gamma(b))^{-1}*\gamma(c)$, and denoting $g_i=(z_i,t_i)$ for $i=1,2$, we get
\begin{align}\label{eq:Npa_est1}
&a\sqrt{|t_1+t_2+2\omega(z_1,z_2)}=a\sqrt{|t(c)|}=N_{p,a}(\gamma(c))=c=b+|c-b|=N_{p,a}(g_1)+N_{p,a}(g_2).
\end{align}
From \eqref{eq:estimate_for_triangle} (formulated for $1<p\leq 2$, the other cases work analogously), we see that \eqref{eq:Npa_est1} implies

\begin{align}\label{eq:Npa_est2}
||z_1||_p ||z_2||_p=N_{p,a}(g_1) N_{p,a}(g_2).
 \end{align}
 Since $g_2\neq e$ (because $\gamma$ is injective), \eqref{eq:Npa_est2} implies $||z_1||_p= N_{p,a}(g_1)$, which by definition of $g_1$ means a contradiction.
 \end{proof}
We remark that all calculations made so far are also true for $p\in\{1,\infty\}$. By continuity, the assertion of Claim 1 can be extended to
\begin{align}\label{eq:ClaimConseq}
\|z(s)\|_p\geq a\sqrt{|t(s)|}, \ \forall s\in[0,1].
\end{align}
This shows in particular, together with the assumption $z(0)=0$ and the injectivity of $\gamma$, that $z(s)\neq0,$ for all $s\in(0,1]$. We can now show a stronger fact, where the assumption $p\in(1,\infty)$ starts to be essential.

\medskip\noindent{\textbf{Claim 2.}}
For all $s\in[0,1]$, there exists $C(s)\in\mathbb{R}$, such that $z(s)=C(s)z(1)$

 \begin{proof}
 Since $z(0)=0$, without loss of generality, we can assume that $s\in(0,1]$. Using $N_{p,a}(\gamma(s))=\|z(s)\|_p$, as established in \eqref{eq:ClaimConseq} for all $s\in [0,1]$, we have
 \begin{align*}
 1=s+|1-s|=N_{p,a}(\gamma(s))+N_{p,a}((\gamma(s))^{-1}*\gamma(1))&\geq\|z(s)\|_p+\|z(1)-z(s)\|_p\nonumber\\
 &\geq\|z(s)\|_p+(\|z(1)\|_p-\|z(s)\|_p)\\&=N_{p,a}({\gamma}(1))=1,\nonumber\\
  \end{align*}
  and hence
  \begin{equation}\label{eq:ConcludedConvexity}
  \|z(s)\|_p+\|z(1)-z(s)\|_p = \|z(1)\|_p.
  \end{equation}
 We know that $z(s)\neq0$. If $z(1)-z(s)=0$, then the assertion of Claim 2 is obviously true. Otherwise, if $z(1)-z(s)\neq0$, it follows from \eqref{eq:ConcludedConvexity} and the strict convexity of the norm $\|\cdot \|_p$ for $p\in(1,\infty)$, that there exists $\alpha(s)\in\mathbb{R}\backslash\lbrace0\rbrace$ such that
 \begin{align}\label{eq:alpha}
 \alpha(s) z(s)=z(1)-z(s).
 \end{align}
 Since we also know that $z(1)\neq0$, it follows from \eqref{eq:alpha} that $\alpha(s)\neq-1$ and $z(s)=(1/(1+\alpha(s)))z(1)$.
 \end{proof}
 Now, we focus our attention on the map $t:[0,1]\to\mathbb{R}$. In order to prove that $\hat{\gamma}$ is the segment of a horizontal line through the origin, we still need to show that this map is actually zero everywhere. The results obtained sofar allow us to  do this: \medskip

\noindent{\textbf{Claim 3.}}
 $t\equiv0.$
 \begin{proof}
 It suffices to show that $t$ is 2-H\"older, which implies that $t$ is constant on $[0,1]$.
  In order to prove the $2$-H\"older continuity of $t$, we first remark that the assertion of Claim 2 in particular implies $\omega_n (z(s_1),z(s_2))=0$, for all  $s_1,s_2\in[0,1]$. Taking this into account, we get
 \begin{align}
 |s_1-s_2|=N_{p,a}((\hat\gamma(s_1))^{-1}*\hat\gamma(s_2))\geq a\sqrt{|t(s_2)-t(s_1)|},\nonumber\\
  \end{align}
  which yields the claim.
  \end{proof}

 Summarizing, what we have got so far is that every geodesic $\gamma:[0,1]\to (\mathbb{H}^1,d_{N_{a,p}})$ for $p\in (1,\infty)$ can be written as $\gamma(s)=(C(s)z(1),0)$, with a vector $z(1)\in\mathbb{R}^{2n}\setminus \{0\}$, and a map $C:[0,1]\to\mathbb{R}$. In particular, this implies that the curve $\gamma_I:([0,1],|\cdot |)\to(\mathbb{R}^{2n},\|\cdot\|_p),\ s\mapsto C(s)z(1)$ is a geodesic through zero and a line segment in $\mathbb{R}^{2n}$. Since $C(1)=1$, it follows that $C$ is the identity map, and hence $\gamma(s)=(sz(1),0)$ for $s\in [0,1]$.

Upon left translation and reparameterization, we have thus shown that every geodesic segment in $(\mathbb{H}^1,d_{N_{a,p}})$ is linear and thus, in consequence, every infinite geodesic in this space is a horizontal line.
 %
%
%
%
%
\end{proof}

We conclude that, unlike for real vector spaces, the properties (horizontal) strict convexity, midpoint property, and {geodesic linearity} are not all equivalent in the Heisenberg group. Both horizontal strict convexity and midpoint property imply the geodesic linearity property, so that the assertion of Theorem \ref{t;main1} remains valid if we replace ``\ldots If every infinite geodesic in $(\mathbb{H}^{n},d')$ is a line\ldots'' by ``\ldots If $(\mathbb{H}^{n},d')$ is horizontally strictly convex\ldots'' or ``\ldots If $(\mathbb{H}^{n},d')$ has the midpoint property\ldots''.


\subsection{Nonlinear embeddings}\label{s:NonLinearEmbedd}

In this section we show through a few examples that for homogeneous distances $d_1$ on $\mathbb{G}\in\{\mathbb{R}^{m},\mathbb{H}^{m}:m\in\mathbb{N}\}$ and $d_2$ on $\mathbb{H}^{n}$, an isometric embedding $f:(\mathbb{G},d_1)\to(\mathbb{H}^{n},d_2)$ does not need to be a homogeneous homomorphism if $d_2$ does not have the geodesic linearity property (GLP). Actually, for the case $\mathbb{G}=\mathbb{R}$ and $d_1(x,y)=|x-y|$, the fact that $d_2$ does not have the GLP already implies, by definition, the existence of an isometric embedding from $\mathbb{G}$ to $\mathbb{H}^{n}$ which is not a homogeneous homomorphism: if $d_2$ does not have the GLP, there exists a geodesic $\gamma:(\mathbb{R},|\cdot|)\to(\mathbb{H}^{n},d_2)$ (hence, in particular an isometric embedding) which is not a horizontal line and clearly such an embedding cannot be a homogeneous homomorphism.
%

Among the examples presented in Section \ref{s:exNorms}, the only two cases not having the GLP are the norms $N_{1,a}$ and $N_{\infty,a}$, for any appropriate positive constant $a$. We justified this assertion arguing that in these two cases $\|\cdot\|_p$ is not strictly convex on $\mathbb{R}^{2m}$ and hence the norm $N_{p,a}$ itself cannot have the GLP either (see Proposition \ref{p:N_pa} and Proposition \ref{p:strictConvexityNecessary}). In fact,  the proof of this implication already provides a method to construct a non linear geodesic $\gamma=(\gamma_I, \gamma_{2n+1})$ with target $(\mathbb{H}^{n},d_{N_{p,a}})$ using a non-linear geodesic $\gamma_I:(\mathbb{R},|\cdot |)\to(\mathbb{R}^{2n},\|\cdot \|_p)$ (whose existence follows from Proposition \ref{p:strict_convex_normed}). In this section we present concrete examples for such geodesics which are not even piecewise linear. In addition we give one example for an isometric embedding $f:(\mathbb{H}^{m},d_1)\to(\mathbb{H}^{n},d_2)$ that is not a homogeneous homomorphism, for the spacial case $d_2=d_{N_{1,a}}$.

\begin{proposition}
The maps
 \begin{align}
       &\gamma:(\mathbb{R},|\cdot |)\to(\mathbb{H}^{n},d_{N_{1,a}}),\quad 0<a\leq 1\nonumber\\
       &\gamma(s):=\left(\tfrac{1}{a}\left(\tfrac{1}{2}(as+\sin(as))e_1+\tfrac{1}{2}(as-\sin(as))e_{n+1}\right),\tfrac{1}{a^2}\left(2\cos(as)+as\sin(as)\right)\right),\nonumber\\
       &\text{and} \nonumber \\
     &\gamma:(\mathbb{R},|\cdot |)\to(\mathbb{H}^{n},d_{N_{\infty,a}}),\quad 0<a\leq \frac{1}{\sqrt{n}}\nonumber\\ &\gamma(s):=\left(\tfrac{1}{a}\left(ase_1+\tfrac{\sin(as)}{2}e_{n+1}\right),\tfrac{1}{a^2}\left(-2\cos(as)-as\sin(as)\right)\right),\nonumber
       \end{align}
 are isometric embeddings which are not homogeneous homomorphisms.
\end{proposition}

\begin{proof}
We first discuss the embedding for $p=1$. Recall that $N_{1,a}((z,0))=\|z\|_1$ for all $z\in \mathbb{R}^{2n}$. Note that the curve $\gamma$ given above for the case $p=1$ is a horizontal lift of the curve $\gamma_I: \mathbb{R} \to \mathbb{R}^{2n}$, defined by
\begin{displaymath}
\gamma_I(s):= \tfrac{1}{a}\left(\tfrac{1}{2}(as+\sin(as))e_1+\tfrac{1}{2}(as-\sin(as))e_{n+1}\right).
\end{displaymath}
Indeed, one finds
\begin{displaymath}
\dot{\gamma}_3(s)= -\tfrac{1}{a}\sin(as) + s\cos(as) = 2\dot{\gamma}_1(s)\gamma_2(s)- 2 \dot{\gamma}_2(s)\gamma_1(s),\quad\text{for all }s\in\mathbb{R}.
\end{displaymath}
According to the proof of Proposition
\ref{p:strictConvexityNecessary}, in order to prove that $\gamma$ is an isometric embedding, it suffices to show that $\gamma_I$ is a geodesic with respect to $\|\cdot\|_1$.
To see this, let us fix $s_1,s_2\in\mathbb{R}$, $s_1\neq s_2$. Then
\begin{align*}
\|\gamma_I(s_2)-\gamma_I(s_1)\|_1&=\tfrac{1}{2a}(|as_2-as_1+(\sin(as_2)-\sin(as_1))|+|as_2-as_1-(\sin(as_2)-\sin(as_1))| \\
&= \tfrac{1}{a}  \max\{|as_2-as_1|,|\sin(as_2)-\sin(as_1)|\}\\
&= |s_2-s_1|,
\end{align*}
by the mean value theorem.

      \medskip

      \noindent In an analogous way, we compute for $p=\infty$ and the respective curve
      \begin{displaymath}
      \dot{\gamma}_3(s)= \tfrac{1}{a}\sin(as)-s\cos(as)= 2 \dot{\gamma}_1(s)\gamma_2(s)-2\dot{\gamma}_s(s)\gamma_1(s),\quad\text{for all }s\in\mathbb{R},
      \end{displaymath}
      and
      \begin{align*}
      \|\gamma_I(s_2)-\gamma_I(s_1)\|_{\infty}&= \max\{|s_2-s_1|,\tfrac{1}{2a}|\sin(as_2)-\sin(as_1)|\}=|s_2-s_1|,
      \end{align*}
      which shows that $\gamma$, which is a horizontal lift of $\gamma_I$, must be an isometric embedding into $(\mathbb{H}^n,d_{N_{\infty,a}})$.
%
 \end{proof}

  \begin{proposition}\label{p:nonlinearEmbed}
For $n\geq 2$ and $0<a\leq \frac{1}{\sqrt{n}}$,  the map
     \begin{align}
f:(\mathbb{H}^{1},d_{N_{\infty,a}})\to(\mathbb{H}^{n},d_{N_{\infty,a}}),\ (x,y
,t)\mapsto(xe_1+\sin(x)e_2+ye_{n+1},t).
\end{align}
is an isometric embedding which is not a homogeneous homomorphism.
\end{proposition}

\begin{proof} Clearly, by Lemma \ref{l:homohHomo}, the map
$f$  is not be a homogeneous homomorphism. On the other hand, we can easily check that $f$ is actually an isometric embedding:
\begin{align}
&d_{N_{\infty,a}}(f(x_1,y_1
,t_1),f(x_2,y_2
,t_2))\nonumber\\
&=\max\{\max\{|x_2-x_1|,|\sin(x_2)-\sin(x_1)|,|y_2-y_1|\},a\sqrt{|t_2-t_1+2x_1y_2-2x_2y_1|}\}\nonumber\\
&=\max\{\max\{|x_2-x_1|,|y_2-y_1|\},a\sqrt{|t_2-t_1+2x_1y_2-2x_2y_1|}\}\nonumber\\
&=d_{N_{\infty,a}}((x_1,y_1,t_1),(x_2,y_2,t_2)),\nonumber
\end{align}
for any $(x_1,y_1,t_1)$ and $(x_2,y_2,t_2)$ in $\mathbb{H}^1$.
\end{proof}

\section{Final comments}\label{s:final}

The a priori information that an isometry or an isometric embedding, if it exists, has to be affine often allows to prove that one space cannot be isometrically embedded into another. As a corollary of Theorem \ref{t:main_Eucl_Heis}, we obtain the following result:

\begin{cor}\label{cor:nonexist}
Let $\mathbb{G}_1\in \{(\mathbb{R}^m,+),(\mathbb{H}^m,\ast)\}$ and $\mathbb{G}_2=(\mathbb{H}^n,\ast)$ for $m\leq n$ be equipped with homogeneous distances $d_1$ and $d_2$, respectively. If $(\mathbb{G}_1,d_1)$ does not have the geodesic linearity property, while $(\mathbb{G}_2,d_2)$ does have the geodesic linearity property, then there cannot exist an isometric embedding $f:(\mathbb{G}_1,d_1)\to (\mathbb{G}_2,d_2)$.
\end{cor}

\begin{proof}
Let us assume towards a contradiction that there exists an isometric embedding $f:(\mathbb{G}_1,d_1)\to (\mathbb{G}_2,d_2)$. By postcomposing with a left translation, we may without loss of generality suppose that $f(0)=0$. Since $(\mathbb{G}_2,d_2)$ has the geodesic linearity property,
Theorem \ref{t:main_Eucl_Heis} yields that $f$ is a homogeneous homomorphism. As $(\mathbb{G}_1,d_1)$ violates the geodesic linearity property, it must contain an infinite geodesic, say $\gamma$, which is not a line. The image $f\circ \gamma$ is a geodesic in $(\mathbb{G}_2,d_2)$ and thus a horizontal line. Yet clearly $f^{-1}|_{f(\mathbb{G}_1)}$ maps lines to lines, so $\gamma$ would have to be a line, which is a contradiction.
\end{proof}

As an application of Corollary \ref{cor:nonexist}, we see immediately by Theorem \ref{t:Npa_GLP} that $(\mathbb{H}^m,d_{N_{p,a}})$ for $p\in \{1,+\infty\}$ does not isometrically embed into $(\mathbb{H}^n,d_{N_{p',a'}})$ for $p'\in (1,+\infty)$. Here the parameters $a$ and $a'$ are chosen so that $N_{p,a}$ and $N_{p',a'}$ are homogeneous norms.

Concerning surjective isometries $f:(\mathbb{H}^n,d_{N_{p,a}})\to (\mathbb{H}^n,d_{N_{p',a'}})$, it follows already from the work of Kivioja and Le Donne that such $f$ must be affine, and in fact it must be a homogeneous homomorphism if we assume, as we may, that $f(0)=0$. The classification of different $\ell_p$ norms on $\mathbb{R}^{2n}$ then yields the isometric classification of the $N_{p,a}$-norms on $\mathbb{H}^n$. As the third author showed in \cite{Sob}, the spaces $(\mathbb{H}^n,d_{N_{p,a}})$ and $(\mathbb{H}^n,d_{N_{p',a'}})$  are isometric exactly in the following cases:
\begin{enumerate}
\item $n=1$, $p=1$, $a=\sqrt{2}b$, $p'=\infty$, $a'=b$, (for $b\in (0,1/\sqrt{2}]$),
\item $n=1$, $p=\infty$, $a=b$,  $p'=1$, $a'=\sqrt{2}b$, (for $b\in (0,1/\sqrt{2}]$),
\item $n\in\mathbb{N}$, $(p,a)=(p',a')$.
\end{enumerate}

As Corollary \ref{cor:nonexist} indicates, it is useful to know whether a homogeneous distance has the geodesic linearity property. In the first Heisenberg group, Corollary \ref{c:plane} reduces the problem  to verifying the strict convexity of a norm in $\mathbb{R}^2$. We conjecture that this works analogously in higher dimensional Heisenberg groups, that is, a homogeneous distance $d_N$ on $\mathbb{H}^n$ has the geodesic linearity property if and only if the norm $\|\cdot\|$ defined by $\|z\|:=N((z,0))$ on $\mathbb{R}^{2n}$ is strictly convex. This conjecture holds true for all the examples considered in this note, and in particular for the norms $N_{p,a}$ from Example \ref{t:Npa_norm}.

\bibliographystyle{amsplain}
\bibliography{references}

\end{document}